\documentclass{article}

\usepackage[a4paper]{geometry}
\usepackage[utf8]{inputenc}
\usepackage{biblatex}
\usepackage{amsmath}
\usepackage{amsthm}
\usepackage{verbatim}
 
\usepackage{xcolor}
\usepackage{amssymb}
\usepackage{dsfont}

\usepackage[citecolor=blue,colorlinks]{hyperref}

\usepackage{enumerate}

\numberwithin{equation}{section}

\newtheorem{definition}{Definition}[section]
\newtheorem{theorem}[definition]{Theorem}

\newtheorem{lemma}[definition]{Lemma}
\newtheorem{remark}[definition]{Remark}
\newtheorem{proposition}[definition]{Proposition}

\def\div{\mathrm{div}}
\def\C{\mathrm{C}}

\def\R{\mathbb{R}}
\def\N{\mathbb{N}}
\def\RCD{\mathrm{RCD}}
\def\CD{\mathrm{CD}}
\def\d{\mathrm{d}}
\def\sfd{\mathsf{d}}

\def\m{\mathfrak{m}}
\def\mm{\mathfrak{m}}
\def\BV{\mathrm{BV}}
\def\tr{\mathrm{tr}}
\def\Per{\mathrm{Per}}
\def\dPer{\d\Per}
\def \weakto{\rightharpoonup}

\def\Xint#1{\mathchoice
   {\XXint\displaystyle\textstyle{#1}}%
   {\XXint\textstyle\scriptstyle{#1}}%
   {\XXint\scriptstyle\scriptscriptstyle{#1}}%
   {\XXint\scriptscriptstyle\scriptscriptstyle{#1}}%
   \!\int}
\def\XXint#1#2#3{{\setbox0=\hbox{$#1{#2#3}{\int}$}
     \vcenter{\hbox{$#2#3$}}\kern-.5\wd0}}

\def\dashint{\Xint-}

\title{Monotonicity formula and stratification of the singular set of perimeter minimizers in RCD spaces}
\author{Francesco Fiorani, Andrea Mondino and Daniele Semola}
\date{}
\begin{document}

\maketitle

\begin{abstract}
  The goal of this paper is to establish a monotonicity formula for perimeter minimizing sets in $\RCD(0,N)$ metric measure cones, together with the associated rigidity statement.\\ The applications include sharp Hausdorff dimension estimates for the singular strata of perimeter minimizing sets in non-collapsed $\RCD$ spaces and the existence of blow-down cones for global perimeter minimizers in Riemannian manifolds with nonnegative Ricci curvature and Euclidean volume growth.
\end{abstract}

\tableofcontents

\section{Introduction}

  The main goal of this paper is to prove a monotonicity formula for perimeter minimizing sets in $\RCD(0,N)$ metric measure cones, together with the associated rigidity statement.\\ Among the applications, we establish sharp Hausdorff dimension estimates for the singular strata of perimeter minimizing sets in non-collapsed $\RCD$ spaces, and the existence of blow-down cones for global perimeter minimizers in Riemannian manifolds with nonnegative Ricci curvature and Euclidean volume growth. 
  
  Below we briefly introduce the setting and then discuss more in detail the main results and their relevance.
  \medskip

The class of $\mathrm{RCD}(K,N)$ spaces consists of infinitesimally Hilbertian metric measure spaces spaces with synthetic Ricci curvature lower bounds and dimension upper bounds. More specifically, $N \in [1,\infty)$ represents a synthetic upper bound on the dimension and $K\in \R$ represents a synthetic lower bound on the Ricci curvature. This class includes finite dimensional Alexandrov spaces with curvature bounded from below and (possibly pointed) measured Gromov Hausdorff limits of smooth Riemannian manifolds with Ricci curvature lower bounds and dimension upper bounds, the so-called Ricci limit spaces. Many of the results in this paper are new also in these settings, to the best of our knowledge. We address the reader to Section \ref{sec:preliminaries} and to the references therein indicated for the relevant background on $\RCD$ spaces.
\smallskip

Sets of finite perimeter have been a very important tool in the developments of Geometric Measure Theory in Euclidean and Riemannian contexts in the last seventy years. In \cite{1BakryEmeryAmbrosio,SemolaGaussGreen,SemolaCutAndPaste}, and the more recent \cite{BrenaGigli,rankOneTheorem}, most of the classical Euclidean theory of sets of finite perimeter has been generalized to $\RCD(K,N)$ metric measure spaces. Moreover in \cite{WeakLaplacian} the second and the third author started a study of locally perimeter minimizing sets in the same setting (see also \cite{GigliMondinoSemola23}). Due to the compactness of the class of $\RCD(K,N)$ spaces with respect to the (pointed) measured Gromov-Hausdorff topology, these developments have been important to address some questions of Geometric Measure Theory on smooth Riemannian manifolds, e.g.\;see  \cite{AntonelliBrueFogagnoloPozzetta,AntonelliPasqualettoPozzettaSemola1}.

\subsection{Monotonicity Formula}

The first main result of this work is a monotonicity formula for perimeter minimizers in cones over $\mathrm{RCD}(N-2,N-1)$ spaces, with the associated conical rigidity statement. We recall that, by \cite{KettererCones}, the metric measure cone over a metric measure space $(X,\sfd,\m)$ is an $\RCD(0,N)$ metric measure space if and only if $(X,\sfd,\m)$ is an $\RCD(N-2,N-1)$ metric measure space.

For the sake of clarity, we introduce below the relevant notion of perimeter minimizing set in an $\RCD(K,N)$ space.

\begin{definition}[Local and Global Perimeter Minimizer]
Let $(X,\sfd,\m)$ be an $\RCD(K,N)$ space. A set of locally finite perimeter $E \subset X$ is a
\begin{itemize}
\item  \emph{Global perimeter minimizer} if  it minimizes the perimeter for every compactly supported perturbation, i.e.
$$
     \Per(E; B_R(x)) \leq \Per(F; B_R(x))
$$
for all $x \in X$, $R>0$ and $F \subset X$ with $F=E$ outside $B_R(x)$;
    \item  \emph{Local perimeter minimizer} if  for every $x\in X$ there exists $r_x>0$ such that $E$ minimizes the perimeter in $B_{r_x}(x)$, i.e. for all $F \subset X$ with $F=E$ outside $B_{r_x}(x)$ it holds
$$
     \Per(E; B_{r_x}(x)) \leq \Per(F; B_{r_x}(x))\, .
$$
\end{itemize}
\end{definition}

Our main result is the following:

\begin{theorem}[Monotonicity Formula]\label{monotonicity formula thm}

Let $N\ge 2$ and let $(X,\sfd,\m)$ be an $\mathrm{RCD}(N-2,N-1)$ space (with ${\rm diam}(X)\le \pi$, if $N=2$). Let $C(X)$ be the metric measure cone over $(X,\sfd,\m)$ and let $O$ denote its tip. Let $E \subset C(X)$ be a global perimeter minimizer. Then the function $\Phi : (0,\infty) \to \R$ defined by
\begin{align}\label{eq:defPhi}
    \Phi(r) := \frac{\Per(E;B_r(O))}{r^{N - 1}}\, ,
\end{align}
is non-decreasing. Moreover, if there exist $0<r_1<r_2<\infty$ such that $\Phi(r_1)=\Phi(r_2)$, then $E\cap \big(B_{r_2}(O)\setminus \overline{B_{r_1}(O)} \big)$ is a conical annulus, in the sense that there exists $A\subset X$ such that 
$$
E\cap \big(B_{r_2}(O)\setminus \overline{B_{r_1}(O)} \big)= C(A) \cap \big(B_{r_2}(O)\setminus \overline{B_{r_1}(O)} \big)\, ,
$$
where $C(A)=\{(t,x)\in C(X)\colon x\in A\}$ is the cone over $A\subset X$.
In particular, if  $\Phi$ is constant on $(0,\infty)$, then $E$ is a cone (in the sense that there exists $A\subset X$ such that $E=C(A)$).
\end{theorem}

The above monotonicity formula with rigidity generalizes the analogous, celebrated result in the Euclidean setting, see for instance \cite{Federerbook,Morganbook,maggi_2012}. On smooth Riemannian manifolds, it is well known that an almost monotonicity formula holds, with error terms depending on two sided bounds on the Riemann curvature tensor and on lower bounds on the injectivity radius. For cones over smooth cross sections, the monotonicity formula is a folklore result, see for instance \cite{Andersonmini,DingJostXin}. Some special cases of Theorem \ref{monotonicity formula thm} have been discussed recently in \cite{Dingarea1,WeakLaplacian}. We also mention that in \cite{FreeboundaryRCD} an analogous monotonicity formula in $\RCD$ metric measure cones has been obtained for solutions of free boundary problems, generalizing a well known Euclidean result.\\
In the proof, we adapt one of the classical strategies in the Euclidean setting. The implementation is of course technically more demanding, in particular for the rigidity part, due to the low regularity of the present context.
\medskip

The relevance of Theorem \ref{monotonicity formula thm} for the applications, that we are going to discuss below, comes from the fact that tangent cones of non-collapsed $\RCD(K,N)$ metric measure spaces $(X,\sfd,\mathcal{H}^N)$ and blow-downs of $\RCD(0,N)$ spaces $(X,\sfd,\mathcal{H}^N)$ with Euclidean volume growth are metric measure cones, see \cite{ConeMetric,GigliNonCollapsed,KettererCones} for the present setting and the earlier \cite{CheegerColding96,CheegerColdingStructure1,burago2001course} for previous results in the case of Ricci limit spaces and Alexandrov spaces.
\medskip

It is an open question whether an almost monotonicity formula holds for perimeter minimizers in general $\RCD(K,N)$ spaces, possibly under the non-collapsing assumption. In particular, we record the following:
\smallskip

\textbf{Open question:} let $(X,\sfd,\mathcal{H}^N)$ be an $\RCD(K,N)$ space and let $E\subset X$ be a local perimeter minimizing set. Is it true that the limit 
\begin{equation}
    \lim_{r\to 0}\frac{\Per(E; B_r(x))}{r^{N-1}}
\end{equation}
exists for all $x\in\partial E$?

\subsection{Stratification of the singular set and other applications}

It is well known that monotonicity formulas are an extremely powerful tool in the analysis of singularities of several problems in Geometric Analysis. We just mention here, for the sake of illustration and because of the connection with the developments of the present work:
\begin{itemize}
\item the Hausdorff dimension estimates for the singular strata of area minimizing currents in codimension one, originally obtained in \cite{Federerstrata};
\item the Hausdorff dimension estimates for the singular strata of $\RCD(K,N)$ spaces $(X,\sfd,\mathcal{H}^N)$, obtained in \cite{GigliNonCollapsed} and earlier in \cite{CheegerColdingStructure1} in the case of non-collapsed Ricci limit spaces.
\end{itemize}
The proofs of the aforementioned results are based on the so-called dimension reduction technique, which relies in turn on the validity of a monotonicity formula with associated conical rigidity statement.\\
In the present work, building on the top of Theorem \ref{monotonicity formula thm} we establish analogous Hausdorff dimension estimates for the singular strata of perimeter minimizing sets in $\RCD(K,N)$ spaces $(X,\sfd,\mathcal{H}^N)$. Below we introduce the relevant terminology and state our main results.
\medskip

\begin{definition}[Singular Strata]\label{definition | singular stratum of a finite perimeter set intro}
Let $(X,\sfd,\mathcal{H}^N)$ be an $\mathrm{RCD}(K,N)$ space, $E\subset X$ a locally perimeter minimizing set and $0 \leq k \leq N-3$ an integer. The $k$-singular stratum of $E$,  $\mathcal{S}^E_k$, is defined as
\begin{equation*}
    \label{singular stratum definition}
    \begin{aligned}
    \mathcal{S}_k^E := &\{x \in \partial E \colon \mbox{no tangent space to $(X,\sfd,\mathcal{H}^N,x,E)$is of the form }(Y,\rho, \mathcal{H}^N, y, F), \\
    &\; \mbox{ with }(Y,\rho,y) \mbox{ isometric to }(Z\times\R^{k+1},\sfd_Z \times \sfd_\mathrm{eucl},(z,0)) \mbox{ for some pointed }(Z,\sfd_Z,z)\\
    &\; \mbox{ and }F=G\times\R^{k+1} \mbox{ with } G\subset Z \mbox{ global perimeter minimizer}\}.
    \end{aligned}
\end{equation*}
\end{definition}

The above definition would make sense also in the cases when $k\ge N-2$. However, it seems more appropriate not to adopt the terminology \emph{singular strata} in those instances.

\begin{definition}[Interior and Boundary Regularity Points]
 Let $(X,\sfd,\mathcal{H}^N)$ be an $\mathrm{RCD}(K,N)$ space and let $E\subset X$ be a locally perimeter minimizing set. Given $x\in \partial E$, we say that $x$ is an \emph{interior regularity} point if
 \begin{equation}
     \mathrm{Tan}_x(X,\sfd,\mathcal{H}^N,E)=\{(\R^N,\sfd_{\mathrm{eucl}},\mathcal{H}^N, 0,  \R^N_+)\}\, .
 \end{equation}
 The set of interior regularity points of $E$ will be denoted by $\mathcal{R}^E$.\\
Given $x\in \partial E$, we say that $x$ is a \emph{boundary regularity} point if
  \begin{equation}
     \mathrm{Tan}_x(X,\sfd,\mathcal{H}^N,E)=\{(\R^N_+,\sfd_{\mathrm{eucl}},\mathcal{H}^N, 0, \{x_1\ge 0\})\}\, ,
 \end{equation}
 where $x_1$ is one of the coordinates of the $\R^{N-1}$ factor in $\R^{N}_+=\R^{N-1}\times \{x_N\ge 0\}$. The set of boundary regularity points of $E$ will be denoted by $\mathcal{R}^E_{\partial X}$.
\end{definition}

It was proved in \cite{WeakLaplacian} that the interior regular set $\mathcal{R}^E$ is topologically regular, in the sense that it is contained in a H\"older open manifold of dimension $N-1$.  By a blow-up argument it is not hard to show that $ \mathrm{dim}_\mathcal{H}\mathcal{R}^E_{\partial X}\leq N-2$  (see Proposition \ref{prop:RECod2}). Our main results about the stratification of the singular set for perimeter minimizers are that the complement of $\mathcal{S}_{N-3}^E$ in $\partial E$ consists of either interior or boundary regularity points, and that the classical Hausdorff dimension estimate $\dim(\mathcal{S}^E_k)\le k$ holds for any $0\le k\le N-3$.

\begin{theorem}\label{thmn: compl n-3 intro}
Let $(X,\sfd,\mathcal{H}^N)$ be an $\mathrm{RCD}(K,N)$ space and let $E\subset X$ be a locally perimeter minimizing set. Then 
\begin{equation}
    \partial E\setminus \mathcal{S}_{N-3}^E=\mathcal{R}^E\cup \mathcal{R}^E_{\partial X}\, .
\end{equation}
\end{theorem}

\begin{theorem}[Stratification of the singular set]\label{thm: strat intro}
Let $(X,\sfd,\mathcal{H}^N)$ be an $\mathrm{RCD}(K,N)$ space and $E\subset X$ a locally perimeter minimizing set. Then, for any $0\le k\le N-3$ it holds
\begin{equation}
    \mathrm{dim}_\mathcal{H} \mathcal{S}_k^E \leq k\, .
\end{equation}
\end{theorem}

We point out that the Hausdorff dimension estimate for the top dimensional singular stratum had already been established in \cite{Dingarea1} (for limits of sequences of codimension one area minimizing currents in smooth Riemannian manifolds with Ricci curvature and volume lower bounds) and independently by the second and the third author in \cite{WeakLaplacian} (in the same setting of the present paper, under the additional assumption that $\partial X=\emptyset$). Elementary examples illustrate that the Hausdorff dimension estimates above are sharp in the present setting.
\smallskip

With respect to the classical \cite{Federerstrata} or \cite{CheegerColdingStructure1,GigliNonCollapsed}, in the proof of Theorem \ref{thm: strat intro} we need to handle the additional difficulty that a monotonicity formula does not hold directly on the ambient space.
\medskip

Another application of the monotonicity formula with the associated rigidity is that if an $\RCD(0,N)$ space $(X,\sfd,\mathcal{H}^N)$  with Euclidean volume growth contains a global perimeter minimizer, then any asymptotic cone  contains a  perimeter minimizing cone. 

\begin{theorem}\label{thm:blowdownintro}
Let $(X,\sfd,\mathcal{H}^N)$ be an $\RCD(0,N)$ metric measure space with Euclidean volume growth, i.e. satisfying for some (and thus for every) $x\in X$:
\begin{equation}\label{eq:AssEVG}
\lim_{r\to \infty} \frac{\mathcal{H}^N(B_r(x))}{r^N} >0.
\end{equation}
Let $E\subset X$ be a global perimeter minimizer. Then for any blow-down $(C(Z),\sfd_{C(Z)},\mathcal{H}^N)$ of  $(X,\sfd,\mathcal{H}^N)$ there exists a cone $C(W)\subset C(Z)$ global perimeter minimizer.
\end{theorem}

The conclusion of Theorem \ref{thm:blowdownintro} above seems to be new also in the more classical case of smooth Riemannian manifolds with non-negative sectional curvature, or nonnegative Ricci curvature. We refer to \cite{Andersonmini} for earlier progress in the case of smooth manifolds with non-negative sectional curvature satisfying additional conditions on the rate of convergence to the tangent cone at infinity and on the regularity of the cross section, and to the more recent \cite{DingJostXin} for the case of smooth Riemannian manifolds with non-negative Ricci curvature and quadratic curvature decay.

\medskip

Related to the open question that we raised above, to the best of our knowledge it is not currently known whether in the setting of Theorem \ref{thm:blowdownintro} any blow down of the perimeter minimizing set must actually be a cone.

\medskip

\textbf{Acknowledgements.}

The first author is supported by the EPSRC-UKRI grant “Maths DTP 2021-22”, at the University of Oxford, with reference EP/W523781/1.

The second author is supported by the European Research Council (ERC), under the European's Union Horizon 2020 research and innovation programme, via the ERC Starting Grant  “CURVATURE”, grant agreement No. 802689.

The last author was supported by the European Research Council (ERC), under the European's Union Horizon 2020 research and innovation programme, via the ERC Starting Grant  “CURVATURE”, grant agreement No. 802689, while he was employed at the University of Oxford until August 2022. He was supported by the Fields Institute for Research in Mathematical Sciences with a Marsden Fellowship from September 2022 to December 2022. He is currently supported by the FIM-ETH Z\"urich with a Hermann Weyl Instructorship. He is grateful to these institutions for the excellent working conditions during the completion of this work.

\section{Preliminaries}\label{sec:preliminaries}

In this work, a \emph{metric measure space} (m.m.s. for short) is a triple $(X,\sfd,\mm)$, where $(X,\sfd)$ is a complete and separable metric space and $\mm$ is a non-negative Borel measure on $X$ of full support (i.e. $\mathrm{supp}\, \mm=X$), called the ambient or reference measure, which is finite on metric balls. We write $B_r(x)$ for the open ball centered at $x\in X$ of radius $r>0$.  Under our working conditions, the closed metric balls are compact, so we assume from the beginning the metric space $(X,\sfd)$ to be proper.
Let $(X,\sfd,\mm)$ be a m.m.s. and fix $x \in X$. The quadruple $(X,\sfd,\mm,x)$ is called \emph{pointed} metric measure space.

We will denote with $L^p(X;\m) := \left \{ u: X \to \R : \int |u|^p \d \mm < \infty\right\}$ the space of $p$-integrable functions; sometimes, if it is clear from the context which space and measure we are considering, we will simply write $L^p$ or $L^p(X)$. 

Given a function $u:X \xrightarrow[]{}\R$, we define its local Lipschitz constant at $x \in X$ by 
$$
    \mathrm{lip}(u)(x) := \limsup_{y\xrightarrow[]{}x} \frac{|u(x)-u(y)|}{\sfd(x,y)}\qquad \mbox{if } x \in X \mbox{ is an accumulation point},
$$
and $\mathrm{lip}(u)(x) = 0$ otherwise.
We indicate by $\mathrm{LIP}(X)$ and $\mathrm{LIP}_\mathrm{loc}(X)$ the space of Lipschitz functions, and locally Lipschitz functions, respectively. We also denote by $\C_b(X)$ and $\C_\mathrm{bs}(X)$ the space of bounded continuous functions, and the space of continuous functions with bounded support, respectively. 
\smallskip

We assume that the reader is familiar with the notion and basic properties of $\RCD$ spaces. Let us just briefly recall \cite{Sturm1, Sturm2, LottVillani} that a $\CD(K,N)$ metric measure space $(X,\sfd,\mm)$ has Ricci curvature bounded below by $K\in \R$ and dimension bounded above by $N\in [1,+\infty]$ in a synthetic sense, via optimal transport.  The $\RCD(K,N)$ condition is a refinement of the $\CD(K,N)$ one, obtained by adding the assumption that the heat flow is linear or, equivalently, that the Sobolev space $W^{1,2}(X,\sfd,\mm)$ is a Hilbert space or, equivalently, that the Laplacian is a linear operator. The $\RCD$ condition was first introduced in the $N=\infty$ case in \cite{Ambrosio_2014} and then proposed in the $N<\infty$ case in \cite{GigliMemo2015}.   We refer the reader to the original papers \cite{Ambrosio_2014, GigliMemo2015, AGMR, EKS15, AMS19, CavallettiMilman}, or to the survey \cite{AmbICM} for more details. 

\subsection{Metric-measure cones}\label{SubSec:MMCones}

An $N$-metric measure cone over a measure metric space  $(X,\sfd_{X}, \m_X)$ is defined as the warped product 
$$C(X):= ([0,\infty) \times_{C} X, \d_{C}, \m_C)$$ obtained with $C_{\sfd} (r) = r^2 $ and $C_\m=r^{N - 1}$. See \cite{Gigli2018SobolevSO} for some background.\\ 
We will denote by $O\in C(X)$ the tip of the cone, given by  $\{O\} : = \pi(\{0\}\times X) \subset C(X)$. In what follows we will use a slight abuse of notation and denote $\pi(t,x)$ by $(t,x)$. In particular, $O = (0,x)$ for any $x \in X$. Moreover, we shall adopt the more intuitive notation $\d_C$, $\m_C$ to denote the distance and the reference measure on $C(X)$, when there is no risk of confusion.

There is an explicit expression for the distance between two points on a cone:
\begin{equation}\label{distance on a cone}
    \sfd^2_C ((r_1,x_1),(r_2,x_2)) = r_1^2 +r_2^2 - 2r_1 r_2 \cos (\sfd_X(x_1,x_2)\wedge \pi).
\end{equation}
In particular, we have
\begin{equation}\label{explicit expression for the distance}
    \sfd_C (O, (r,x))=r\, .
\end{equation}

The following result was obtained in \cite{ConeMetric}.
\begin{proposition}\label{metric cone definition}
Let $(X,\sfd_X,\m_X)$ be a metric measure space and let $N\ge 2$. Then $(C(X), \sfd_C, \m_C)$ is an $\mathrm{RCD}(0,N)$ m.m.s. if and only if $(X,\sfd_X,\mm_X)$ is $\mathrm{RCD}(N-2,N-1)$ and, in the case $N=2$, $\mathrm{diam}(X)\le \pi$.
\end{proposition}

\begin{remark}
If $N>2$, then the diameter bound $\mathrm{diam}(X)\le \pi$ follows already from the $\RCD(N-2,N-1)$ condition, by the Bonnet-Meyers theorem for $\CD$ spaces.   
\end{remark}

Next, we show a result relating radial derivatives of functions defined on cones and the distance from the tip of the cone. The gradient of such distance function, in some sense, corresponds to the position vector field in the Euclidean setting. Some of the identities of the proof, obtained using basic calculus in the metric setting, will be used later in this work. The following result uses the BL characterization found in \cite{Gigli2018SobolevSO}, section 3.2.

\begin{proposition}\label{characterization of scalar product}
Let $(C(X), \sfd_C, \m_C)$ be an $\RCD(0,N)$ cone over some $\RCD(N-2,N-1)$ space $(X,\d_X,\m_X)$ and $f \in W^{1,2}(X)$. Let $f^{(t)}(x) := f(t,x)$ and $f^{(x)}(t) := f(t,x)$, for $(t,x) \in C(X)$. Then
\begin{equation}
    |\nabla f^{(x)}| (t) =\frac{1}{2t}| \nabla f (t,x) \cdot \nabla \sfd_C^2 (O, \cdot) (t,x)|\, \quad \text{for $\m$-a.e. $x\in X$ and $\mathcal{L}^1$-a.e. $r\in(0,\infty)$}\, .
\end{equation}
\begin{proof}
By \cite{Gigli2018SobolevSO}, section 3.2, $f^{(t)} \in W^{1,2}(X)$ and $f^{(x)} \in W^{1,2}([0,\infty))$ for $\mathcal{L}^1$-a.e.\;$t \in (0,\infty)$ and $\m$-a.e.\;$x \in X$, respectively. By using the BL characterization of $W^{1,2}(C(X))$ functions (cf. \cite{Gigli2018SobolevSO}, section 3.2) and the polarization identity, we have 
    \begin{align}
    \nabla f (t,x)&\cdot \nabla \sfd_C^2(O, \cdot) (t,x)= \nonumber \\
    = &\frac{1}{2} |\nabla (f +\sfd_C^2(O,\cdot))|^2(t,x) - \frac{1}{2} |\nabla(f-\sfd_C^2(O,\cdot))|^2 (t,x)\nonumber \\
    = &\frac{1}{2} |\partial_r (f^{(x)}(t) +\sfd_{C,(x)}^2(t))|^2 - \frac{1}{2} |\partial_r (f^{(x)}(t)-\sfd_{C,(x)}^2(t))|^2\label{third equality}\\
    =&\frac{1}{2} |\partial_r f^{(x)} +2t|^2 - \frac{1}{2} |\partial_r f^{(x)}(t)-2t|^2\label{fourth equation}\\
    =& 2t\, \partial_r f^{(x)}(t)= 2t \, \mathrm{sign}(\partial_r f^{(x)}(t))\, |\nabla f^{(x)}|(t)\, ;\nonumber
    \end{align}
where we have denoted $\sfd_C(O,(t,x))^{(x)}$ by $\sfd_{C,(x)}(t)$. Moreover, we have used the fact that $|\nabla (f^{(t)} +(\sfd_C^2(O,\cdot))^{(t)})|^2(x)= |\nabla f^{(t)}|(x)$ since $(\sfd_C(O,\cdot))^{(t)}$ is constant. We have also exploited the identification between different notions of derivatives (by, say, \cite[Th.\;2.1.37]{GigliPasqualetto}) $|\nabla g|=|\partial_rg|$ for smooth functions $g:[0,\infty)\to \R$ and the explicit formula for the radial sections of the distance function from the origin given by \eqref{explicit expression for the distance}.
Let us also point out that \eqref{third equality} shows that
\begin{equation}\label{radial scalar product}
    \nabla f (t,x)\cdot \nabla (\sfd_C^2(O,(t,x))  = \nabla f^{(x)}(t)\cdot \nabla (\sfd^2_{C, (x)})(t)\, .
\end{equation}
Lastly, we point out that the following equality
\begin{equation}\label{distance squared to distance}
    |\nabla  \sfd_C(O,\cdot)|(t,x) = |\partial_r \sfd_{C,(x)} (t)| = \frac{1}{2t} |\nabla \sfd_C^2(O,\cdot)|(t,x)\, ,
\end{equation}
implies
\begin{equation}\label{scalar product characteerization without square}
    |\nabla f^{(x)}| (t) =| \nabla f (t,x) \cdot \nabla (\sfd_C (O, (t,x)))|\, .
\end{equation}
See \cite[Cor.\;3.6]{ConeMetric}.
\end{proof}
\end{proposition}

\subsection{Finite perimeter sets in $\mathrm{RCD}$ spaces}\label{section| preliminaries: finite perimter sets RCD}

Let $(X,\sfd,\mm)$ be a metric measure space,  $u \in L^1_{\mathrm{loc}}(X)$ and  $\Omega \subset X$ an open set. The total variation norm of $u$ evaluated on $\Omega$ is defined by
\begin{equation}\label{definition of total variation measure}
|D u|(\Omega) := \inf \left \{\liminf_{j \xrightarrow[]{}\infty} \int_\Omega \mathrm{lip}(u)(y) \, \d \m \right \}\, ,
\end{equation}
where the infimum is taken over all sequences $(u_j) \subset \mathrm{Lip}_{\mathrm{loc}}(X)$ such that $u_j \xrightarrow[]{} u$ in $L^1_{\mathrm{loc}}(X)$. 

A function $u \in L^1(X)$ is said to have bounded variation if its total variation $|Du|(X)$ is finite. In this case one can prove that $|Du|$ can be extended to a Borel measure on $X$. The space of functions of bounded variation is denoted by $\BV(X)$.

A set $E\subset X$ is of locally finite perimeter if, for all $x \in X$ and $R>0$ there holds
\begin{equation*}
    \mathrm{Per}(E; B_R(x)): = \inf \left\{\liminf_{i \to \infty} \int_{B_r(x)} \mathrm{lip} (f_i) \, \d\m\, : \, \{f_i\} \subset \mathrm{Lip}_\mathrm{loc}(X), \, f_i \overset{L^1_{loc}}{\longrightarrow} \chi_E \right\} < \infty\, .
\end{equation*}

We denote the perimeter measure of a locally finite perimeter set $E $ by $\mathrm{Per}(E)$. Let us point out that this coincides with the variational measure $|D\chi_E|$ defined in \eqref{definition of total variation measure}. We adopt the following notation: given a Borel set $A \subset X$, $\mathrm{Per}(E;A):= |D\chi_E(A)|$.

An important tool for what follows is the coarea formula (cf. \cite{mir00}, theorem 2.16):

\begin{theorem}[Coarea Formula]
Let $(X,\sfd,\mm)$ be an $\mathrm{RCD}(K,N)$ space and $v \in \BV(X)$. Then $\{v>r\}$ has finite perimeter for $\mathcal{L}^1$-a.e.\;$r$ and %
for any Borel function $f :X \to \R$ it holds 
\begin{equation}
    \label{coarea formula 1}
    \int_X f \,\d |Dv| = \int_{\R} \left ( \int_X f\, \dPer(\{v>r\}) \right) \d r\, ,
\end{equation}
where the term $\d|Dv|$ in the left hand side denotes that the integral is computed with respect to the total variation measure $|Dv|$.
\end{theorem}

Given a function $u\in \BV(X)$, we define the set $E_t := \{x \in X: u(x) \geq t$, $t \in \R\}$. As a consequence of the coarea formula, it holds
$$
    |Du|(C) = \int_\R \Per(E_t; C) \,\d t\, .
$$

We now look at a particular type of locally finite perimeter sets: cones inside cones. Let $(C(X), \sfd_C, \m_C)$ be as in Proposition \ref{metric cone definition}, $r>0$, $x \in X$. An important family of sets we will use later, in particular in the characterization of cones (see Lemma \ref{characterization of cones}), is the following
\begin{equation}\label{cone over ball inside a cone}
    C(B^X_r(x)):=\{(t,y) \in C(X): y \in  B^X_r(x)\}\, .
\end{equation}

\begin{lemma}\label{locally finite perimeter of a cone over a set}
The sets $C(B^X_r(x)) \subset C(X)$, $r>0$, are of locally finite perimeter.
\begin{proof}
We will prove the claim by exhibiting an explicit sequence of Lipschitz functions converging to $\chi_{C(B^X_r(x))}$. Let 
\begin{equation}
    \label{sequence approximating the cone over a ball}
    f_n(t,y) = \begin{cases}
    1 & \mbox{if } y \in B^X_r(x)\, ; \\
    nr +1 - n\sfd(y,x) & \mbox{if } y \in B^X_{r + \frac{1}{n}}(x)\setminus B^X_r(x)\, ; \\
    0 & \mbox{if } y \in X \setminus B^X_{r+\frac{1}{n}}(x))\, . 
    \end{cases}
\end{equation}
Clearly, $f_n$ is Lipschitz with bounded support, $|f_n| \leq 1$, hence $f_n \in W^{1,2}_\mathrm{loc}(C(X))$. Moreover, $f_n \to \chi_{C(B^X_r(x))}$ in $L^1_\mathrm{loc}(C(X))$. Indeed, for $p=(s,z) \in C(X)$, $R>0$
\begin{align*}
\int_{B_R(p)} |f_n - \chi_{C(B^X_r(x))}|\, \d\m_C &=  \int_{B_R(p)} |f_n| \chi_{C(B^X_{r + \frac{1}{n}}(x)\setminus B^X_r(x))}\d\m_C \\
&\leq \m_C\left( B_R(p)\cap C\left(B^X_{r + \frac{1}{n}}(x)\setminus B^X_r(x)\right)\right).
\end{align*}
 Using the Bishop Gromov inequality \cite{Sturm2}, we have 
$$
    \m(B^X_{r + \frac{1}{n}}(x)\setminus B^X_r(x)) = \m(B^X_r(x))\left(\frac{\m(B^X_{r + \frac{1}{n}}(x)}{\m(B^X_r(x))} - 1\right) 
\leq \m(B^X_r(x))\left(\frac{N-1}{rn} + O\left(\frac{1}{n^2}\right) \right)\, .
$$
Therefore, 
\begin{equation}\label{eq:momegaBR}
    \m_C\left(B_{R}(p) \cap C\left(B^X_{r + \frac{1}{n}}(x)\setminus B^X_r(x)\right)\right) = O\left(\frac{1}{n}\right)\, ,
\end{equation}
which implies  $f_n \to \chi_{C(B^X_r(x))}$ in $L^1_\mathrm{loc}(C(X))$.

Let us now show that 
\begin{equation}
\limsup_{n \to \infty}\int_{B_R(p)}\mathrm{lip}f_n \d\m_C < \infty\, ,
\end{equation}
for any $p=(s,z) \in C(X)$, $R>0$, which directly implies that $C(B^X_r(x))$ is a set of locally finite perimeter. 

It is elementary to check that
\begin{equation}
    \label{derivative of approximants to the cone over a ball}
    \mathrm{lip}f_n(t,x) = \begin{cases}
    n & \mbox{if } y \in B^X_{r + \frac{1}{n}}(x)\setminus B^X_r(x)\, ;\\
    0 & \mbox{otherwise}\, .
    \end{cases}
\end{equation}
Therefore, using \eqref{eq:momegaBR}, we obtain
\begin{align*}
\int_{B_R(p)}\mathrm{lip}f_n \d\m_C= n\; \m_C\left(B_{R}(p) \cap C\left(B^X_{r + \frac{1}{n}}(x)\setminus B^X_r(x)\right)\right)= O(1)\, , \quad \text{as } n\to \infty\, .
\end{align*}
\end{proof}
\end{lemma}

Let us recall some useful notions of convergence, in order to study blow-ups of sets of (locally) finite perimeter. We refer to \cite{Gigli_2015, AmbrosioHonda1, 1BakryEmeryAmbrosio} for more details.

\begin{definition}[$L^1$ strong convergence of sets]\label{definition | L^1 convergence of sets}
Let $\{(X_i,\sfd_i,\m_i,x_i)\}_i$ be a sequence of pointed metric measure spaces converging in the pointed measured Gromov Hausdorff sense to $(Y,\rho,\mu, y)$. Let $(Z,\sfd_Z)$ be the ambient space realizing the convergence. Moreover, let $E_i\subset X_i$ be a sequence of Borel sets with $\m_i(E_i)<\infty$ for every $i\in\N$. We say that $\{E_i\}_i$ converges to a Borel set $F\subset Y$ in the strong $L^1$ sense if the measures $\chi_{E_i}\m_i\weakto\chi_F\mu$ in duality with $\C_\mathrm{bs}(Z)$ and $\m_i(E_i)\to \mu(F)$. 

We also say that the convergence of $E_i$ is strong in $L^1_\mathrm{loc}$ if $E_i\cap B_r(x_i)$ converges in the strong $L^1$ sense to $F\cap B_r(y)$ for all $r>0$.
\end{definition}

Such a convergence can be metrized, by a distance $\mathcal{D}$ defined on (isomorphsim classes of) quintuples $(X,\sfd,\m,x,E)$, see \cite[Lemma A.4]{1BakryEmeryAmbrosio}.

We use that notion of convergence to define the tangent to a set of locally finite perimeter contained in an $\mathrm{RCD}(K,N)$ metric measure space. Before, let us recall that given an $\mathrm{RCD}(K,N)$ and a point $x\in X$, by Gromov pre-compactness theorem there is a not-empty set $\mathrm{Tan}_x(X)$ of tangent spaces at $x$ obtained by considering the pmGH limits of blow-up rescalings of $X$ centered at $x$; moreover, for $\mm$-a.e. $x\in X$, the tangent space is unique and Euclidean \cite{GMR15, MondinoStructureTheory, ConstancyDimensionBrueSemola}.
\begin{definition}[Tangent to a set of locally finite perimeter]\label{tangent to a set of locally finite perimeter}\label{definition | Tangent to a set of locally finite perimeter}
Let $(X,\sfd,\m)$ be an $\mathrm{RCD}(K,N)$ m.m.s. and $E\subset X$ be a set of locally finite perimeter.\\ 
We say that $(Y, \rho, \mu, y, F)\in \mathrm{Tan}_x(X,\sfd,\m,E)$ if $(Y,\rho,\mu,y) \in \mathrm{Tan}_x(X)$ and $F\subset Y$ is a set of locally finite perimeter of positive measure such that $\chi_E$ converges in the $L^1_\mathrm{loc}$ sense of Definition \ref{definition | L^1 convergence of sets} to $F$ along the blow up sequence associated to the tangent $Y$.
\end{definition}

An important tool for our analysis is the Gauss-Green formula for sets of finite perimeter in the $\mathrm{RCD}$ setting. We refer to \cite{SemolaGaussGreen} for the proof and for background material; here let us briefly mention that, given a set of finite perimeter $E\subset X$, it is possible to define the space of $L^2$-vector fields with respect to the perimeter measure, denoted by $L^2_E(TX)$.  

\begin{theorem}[Gauss-Green formula]\label{thm:NuEGG}
Let $(X,\sfd,\m)$ be an $\mathrm{RCD}(K,N)$ m.m.s. and $E\subset X$ a set of finite perimeter with $\m(E) < \infty$. Then there exists a unique vector field $\nu_E \in L^2_E(TX)$ such that $|\nu_E| =1$ $|D\chi_E|$-almost everywhere and 
\begin{equation*}
    \int_E \mathrm{div}(v) \d\m = - \int \mathrm{tr}_E (v) \cdot \nu_E\, \dPer(E),
\end{equation*}
for all $v \in W^{1,2}_C(TX) \cap D(\mathrm{div})$ with $|v|\in L^\infty(|D\chi_E|)$.
\end{theorem}

We next recall a useful cut and paste result proved in \cite[Theorem 4.11]{SemolaCutAndPaste}. Do this aim, let us first fix some notation. We denote by $\mathcal{H}^h$ the codimension one Hausdorff type measure induced by $\m$ with gauge function $h(B_r(x)):=\m(B_r(x))/r$, see \cite{SemolaGaussGreen} for further details. We write the total variation measure $D\chi_E$ using the polar factorization  $D\chi_E = \nu_E |D\chi_E|$.  Lastly, we denote the  \emph{measure theoretic interior} $E^{(1)}$ (resp. the  \emph{measure theoretic exterior} $E^{(0)}$) of a Borel subset $E\subset X$:
\begin{align*}
E^{(1)} &:= \left\{x \in X: \lim_{r \to 0} \frac{\mm(E\cap B_r(x))}{\mm(B_r(x))} = 1\right\}, \\
E^{(0)} &:= \left\{x \in X: \lim_{r \to 0} \frac{\mm(E\cap B_r(x))}{\mm(B_r(x))} = 0\right\}.
\end{align*}
\begin{theorem}[Cut and paste]\label{Thm:CutAndPastePer}
Let $(X,\sfd,\m)$ be an $\mathrm{RCD}(K,N)$ m.m.s. and $E$, $F \subset X$ be sets of finite perimeter. Then $E\cap F$, $E\cup F$ and $E\setminus F$ are sets of finite perimeter. Moreover, there holds
\begin{align*}
    D \chi_{E\cap F} &= D\chi_E |_{F^{(1)}} + D\chi_F |_{E^{(1)}} + \nu_E \mathcal{H}^{h}|_{\{\nu_E=\nu_F\}}\, ;\\
     D\chi_{E\cup F} &= D\chi_E |_{F^{(0)}} + D\chi_F |_{E^{(0)}} + \nu_E \mathcal{H}^{h}|_{\{\nu_E=\nu_F\}}\, ;\\
     D\chi_{E\setminus F} &= D\chi_E |_{F^{(0)}} - D\chi_F |_{E^{(1)}} + \nu_E \mathcal{H}^{h}|_{\{\nu_E=-\nu_F\}}\, .
\end{align*}
\end{theorem}

We now recall the notion of  perimeter minimizing sets. To avoid  discussing trivial cases, we will always assume $\mm(E)>0$ and $\mm(X\setminus E)\neq 0$.

\begin{definition}[Local and Global Perimeter Minimizer]\label{def:LocPerMin}
A set of locally finite perimeter $E \subset X$, with $\mm(E)>0$ and  $\mm(X\setminus E)>0$,  is a
\begin{itemize}
\item  \emph{Global perimeter minimizer} if  it minimizes the perimeter for every compactly supported perturbation, i.e.
$$
     \Per(E; B_R(x)) \leq \Per(F; B_R(x))
$$
for all $x \in X$, $R>0$ and $F \subset X$ with $F=E$ outside $B_R(x)$;
    \item  \emph{Local perimeter minimizer} if  for every $x\in X$ there exists $r_x>0$ such that $E$ minimizes the perimeter in $B_{r_x}(x)$, i.e. for all $F \subset X$ with $F=E$ outside $B_{r_x}(x)$ it holds
$$
     \Per(E; B_{r_x}(x)) \leq \Per(F; B_{r_x}(x))\, .
$$
\end{itemize}
\end{definition}

For a proof of the following density result see \cite{KKLSDensity}.

\begin{lemma}[Density]\label{lemma | density of perimeter minimizers}
Let $(X,\d,\m)$ be an $\RCD(K,N)$ m.m.s. and let $E \subset X$ be a local perimeter minimizer set and let $x\in \partial E$. Then there exists constants $r_0$, $C >0$ such that
\begin{equation*}
    C^{-1} \frac{\m(E \cap B_r(x))}{r} \leq \mathrm{Per}(E; B_r(x)) \leq C \frac{\m(E \cap B_r(x))}{r}\, ,
\end{equation*}
for all $0<r<r_0$.
\end{lemma}

We report here a few results on $\mathrm{BV}$ functions and their associated vectorial variational measures, for their proof and background on notation see  \cite{BrenaGigli}.
\begin{proposition}
Let $(X, \sfd, \m)$ be an $\mathrm{RCD}(K,\infty)$ space and $f \in \BV(X)$. Then there exists a unique vector field $\nu_{f}\in L_{|Df|}(TX)$ such that 
\begin{equation}
    \label{vector field BV function}
    \int_X f\, \div v \, \d\m = - \int_X v \cdot \nu_f \, \d|Df|\, , \quad \text{for all $v \in QC^\infty(TX) \cap D(\div)$}.
\end{equation}

\end{proposition}
In what follows, for $f \in \BV(X)$ we will denote $Df := \nu_f |Df|$. 
If $E$ is a set of locally finite perimeter, we denote $\nu_E := \nu_{\chi_E}$. This definition of unit normal is consistent with the one introduced above via the Gauss-Green formula, see \cite{SemolaGaussGreen}. For a function $f:X\to\R$, define
\begin{align*}
    f^\wedge = \mathrm{ap} \ \liminf_{y\to x} f(y) = \sup \left\{ t \in \overline{\R}: \lim_{r \downarrow 0} \frac{\m(B_r(x) \cap \{f<t\})}{\m(B_r(x))}=0 \right\} \\
    f^\vee = \mathrm{ap} \ \limsup_{y\to x} f(y) = \inf \left\{ t \in \overline{\R}: \lim_{r \downarrow 0} \frac{\m(B_r(x) \cap \{f>t\})}{\m(B_r(x))}=0 \right\}\, ,
\end{align*}
and, lastly,
\begin{equation}\label{eq:defbarf}
\Bar{f} = \frac{f^\wedge+f^\vee}{2}\, ,
\end{equation}
with the convention $\infty -\infty = 0$.
\begin{lemma}[Leibniz rule for BV]
Let $(X, \sfd, \m)$ be an $\mathrm{RCD}(K,\infty)$ space and $f, g \in \BV (X) \cap L^\infty(X)$. Then $fg \in \BV(X)$ and 
\begin{equation}
    \label{Leibniz Rule BV}
    D(fg) = \Bar{f} Dg + \Bar{g} Df\, .
\end{equation}
In particular, $|D(fg)| \leq |\Bar{f}| |Dg| + |\Bar{g}| |Df|$.
\end{lemma}

\begin{proposition}[$\BV$ extension]    \label{extension of BV}
Let $(X, \sfd, \m)$ be an $\mathrm{RCD}(K,\infty)$ space, $E$ a set of locally finite perimeter and $f\in \BV(X)\cap  L^\infty(E)$. Then 
$$
\tilde{f} (x) := 
\begin{cases}
\Bar{f}(x)  \quad & \mathrm{if} \ x \in E\, ,\\
0 & \mathrm{elsewhere}
\end{cases}
$$
belongs to $\BV(X)$ and $D\tilde{f} = \Bar{f} D\chi_E + Df|_E$.
\begin{proof}
The result immediately follows by applying \eqref{Leibniz Rule BV} with $g = \chi_E$. 
\end{proof}
\end{proposition}

\begin{lemma}[Cut and paste of $\BV$ Functions]\label{BV cut and paste}
Let $(X, \sfd, \m)$ be an $\mathrm{RCD}(K,\infty)$ m.m.s. and $E$ a set of locally finite perimeter. Let $f \in \BV(E)$ and $g \in \BV(X\setminus E)$. Let $h:X \to \R$ be defined as
\begin{align*}
h(x):=
\begin{cases}
f(x) \quad \mbox{if } x \in E\, ; \\
g(x) \quad \mbox{if } x \in X\setminus E\, . 
\end{cases}
\end{align*}
Then, $h \in \BV(X)$. Moreover, called $\bar{f}, \bar{g}$ the representatives given by \eqref{eq:defbarf}, it holds
$$
    Dh = Df|_E +Dg|_{X\setminus E} + (\Bar{f} - \Bar{g})D\chi_E.
$$
\begin{proof}
Let $\tilde{f}$ and $\tilde{g}$ be the extensions by zero given by Proposition \ref{extension of BV}. Then $h = \tilde{f} + \tilde{g}$.
\end{proof}
\end{lemma}

\section{Monotonicity Formula}

A classical and extremely powerful tool for studying sets which locally minimize the perimeter in Euclidean spaces is the monotonicity formula for the perimeter. The goal of this section is to generalize such monotonicity formula (with the associated rigidity statement) to perimeter minimizers  in cones over $\mathrm{RCD}$  spaces. In the next section, we will draw some applications on the structure of the singular set of local perimeter minimizers.

Recall that given an $\mathrm{RCD}(N-2,N-1)$ space  $(X,\sfd_X,\m_X)$ then the metric-measure cone over $X$, denoted by $(C(X), \sfd_C, \m_C)$, is an  $\mathrm{RCD}(0,N)$ space (if $N=2$, we also assume that $\mathrm{diam}(X)\le \pi$). We denote by $O=(0,x)\in C(X)$  the tip of the cone (see Section \ref{SubSec:MMCones} for more details) and $B_r(O)$ the open metric ball centered at $O$ of radius $r>0$. 

When we consider a local perimeter minimizer $E$, we shall always assume that $E=E^{(1)}$ is the open representative, given by the measure theoretic interior. See \cite{KinunnenJGA2013} for the relevant background.

\begin{theorem}[Monotonicity Formula]\label{Thm:MonotonicityBody}
Let $N\ge 2$ and let $(X,\sfd,\m)$ be an $\mathrm{RCD}(N-2,N-1)$ space (with ${\rm diam}(X)\le \pi$, if $N=2$). Let $C(X)$ be the metric measure cone over $(X,d,\m)$. Let $E \subset C(X)$ be a global perimeter minimizer in the sense of Definition \ref{def:LocPerMin}. Then the function $\Phi : (0,\infty) \to \R$ defined by
\begin{align}\label{eq:defPhi}
    \Phi(r) := \frac{\Per(E;B_r(O))}{r^{N - 1}}\, ,
\end{align}
is non-decreasing. Moreover, if there exist $0<r_1<r_2<\infty$ such that $\Phi(r_1)=\Phi(r_2)$, then $E\cap \big(B_{r_2}(O)\setminus \overline{B_{r_1}(O)} \big)$ is a conical annulus, in the sense that there exists $A\subset X$ such that 
$$
E\cap \big(B_{r_2}(O)\setminus \overline{B_{r_1}(O)} \big)= C(A) \cap \big(B_{r_2}(O)\setminus \overline{B_{r_1}(O)} \big)\, ,
$$
where $C(A)=\{(t,x)\in C(X)\colon x\in A\}$ is the cone over $A\subset X$.
In particular, if  $\Phi$ is constant on $(0,\infty)$, then $E$ is a cone (in the sense that there exists $A\subset X$ such that $E=C(A)$).
\end{theorem}

\begin{remark}
In the case where $E \subset C(X)$ is a locally finite perimeter set,  minimizing the perimeter for perturbations supported in $B_{R+1}(O)$, then the monotonicity formula holds on $(0,R)$, i.e. the function $\Phi$ defined in \eqref{eq:defPhi} is non-decreasing on $(0,R)$. Also the rigidity statement holds, for $0<r_1<r_2<R$. The proofs are analogous.
\end{remark}

\begin{proof}[Proof of Theorem \ref{Thm:MonotonicityBody}]
Let us first give an outline of the argument. The first two steps are inspired by the approach used in the lecture notes \cite{MooneyLectureNotes}, which provide a proof of the monotonicity formula for local perimeter minimizers in Euclidean spaces by-passing the first variation formula. Classical references for this approach are \cite{Federerbook,Morganbook}.\\
The main idea is to approximate the characteristic function of $E$ by regular functions $f_k$ and approximate $\Phi$ by the corresponding $\Phi_{f_k}$; show an almost-monotonicity formula for $\Phi_{f_k}$ and finally pass to the limit and get the monotonicity of $\Phi$. This will be achieved in steps 1-3. 
In step 4 we relate the derivative of $\Phi$  with a quantity characterizing cones as in Lemma \ref{characterization of cones}.

Throughout the proof, we will write $B_r$ in place of $B_r(O)$ for the ease of notation.
\medskip

\textbf{Step 1: Approximation preliminaries.} \\ In this step we show that, up to error terms, regular functions approximating $\chi_E$ preserve the minimality condition. The argument requires an initial approximation. Let $f \in \mathrm{LIP}(C(X))\cap {\rm D}_{\rm loc}(\Delta)(C(X))$ be non-negative. We introduce two functions  $a$, $b: [0,\infty) \xrightarrow[]{} [0,\infty)$ to quantify the errors in the approximation:
\begin{align}\label{errors}
    & a(r) := \left| |D f|(B_r) - \mathrm{Per}(E;B_r) \right|\, , & b(r) := \int_{\partial B_r} |\tr_{\partial B_r}^{\rm ext} \chi_E - \tr_{\partial B_r} f|\,  \dPer{(B_r)}\, ,
\end{align}
where  $\tr_{\partial B_r}^{\rm  ext} \chi_E$ is the trace of $\chi_E$ from the exterior of the  ball $B_r$.\\
We remark that the interior and exterior normal traces can be defined by considering the precise representative of $\chi_E\cdot\chi_{B_r}$ and $\chi_E\cdot\chi_{X\setminus B_r}$ respectively. See \cite[Lemma 3.23]{BrenaGigli} and \cite{AmbrosioFuscoPallara} for the Euclidean theory.

Notice that $\tr_{\partial B_r}^{\rm  ext} f= \tr_{\partial B_r} f  = f|_{\partial B_r}$, since $f$ is continuous. 
Fix $R>0$.  Let $0<r<R$ and $g \in \BV_\mathrm{loc}(C(X))$ be any function such that 
$$\tr_{\partial B_r}^{\rm  int} g = \tr_{\partial B_r} f\quad  \text{and} \quad  g = \chi_E \text{ on } C(X)\setminus B_r\, ,$$
 where  $\tr_{\partial B_r}^{\rm  int} g$ is the trace of $g$ from the interior of the ball $B_r$. The minimality of $E$ implies
\begin{align*}
    \Per(E; B_R) \leq \Per(\{q \in C(X) \, :\,  g(q) > t \}; B_R),
\end{align*}
for any $0<t<1$. Integrating in $t$ and using the coarea formula \eqref{coarea formula 1} we obtain
\begin{align*}
    \Per(E; B_R) \leq \int_0^1 \Per(\{q \in C(X) \, :\,  g(q) > t \}; B_R) \, \d t \leq  |D g| (B_R)\, .
\end{align*}
Therefore, using Lemma \ref{BV cut and paste} and the definition of $g$, we obtain 
\begin{align*}
    \Per(E;B_r) & = \Per(E;B_R)-\Per(E;B_R \setminus B_r) \leq |D g| (B_R) - \Per(E;B_R \setminus B_r) \\
    &= |D g| (B_r) + \int_{\partial B_r} |\tr_{\partial B_r}^{\rm ext} \chi_E - \tr_{\partial B_r} f| \, \dPer(B_r) = |D g| (B_r) + b(r)\, .
\end{align*}
Finally, for any such $g$ there holds
\begin{align} \label{almost minimizer for its perimeter}
    |D f| (B_r) \leq \Per(E;B_r) + a(r) \leq |D g| (B_r) +a(r) + b(r)\, .
\end{align}
\smallskip

\textbf{Step 2. Main computation. }\\
In this step we show the monotonicity, up to error terms, of an approximation  of $\Phi$, denoted below by $\Phi_f$, obtained by replacing $\chi_E$ with the regular approximation $f$ of step 1. 

Fix $f$ as in step 1 and $r>0$. By \cite{Gigli2018SobolevSO},
$$|\nabla f |^2 (t,x) = |\nabla  f^{(x)}|^2(t) + t^{-2}|\nabla f^{(t)}|^2(x)\, , \text{for $\m$-a.e. $x\in X$ and $\mathcal{L}^1$-a.e. $t>0$}\, .$$
Let $h:C(X)\xrightarrow[]{}\R$ be defined by $h(t,x):=f^{(r)}(x)$ for all $t> 0$. Notice that $h$ is locally Lipschitz away from the origin and it is elementary to check that it has locally bounded variation.\\ By \cite{Gigli2018SobolevSO,GIGLIIndependenceOnP}, it holds
\begin{align*}
    |D h|(t,x) = \frac{r}{t} |\nabla f^{(r)}|(x)\, ,\quad \text{for $\m$-a.e. $x$ and  $\mathcal{L}^1$-a.e. $t$}\, .
\end{align*}
By integrating over $B_r$ and using the coarea formula, we obtain
\begin{align}
     \int_{B_r}  |D h| (t,x) \, \d\m_C &= \int_0^r \int_{\partial B_t} |D h| (t,x) \, \dPer( B_t)\,dt 
    = \int_0^r t^{N - 1} \int_{X} \frac{r}{t} |\nabla f^{(r)}|(x) \d\m \, \d t \nonumber \\
    &= \int_0^r \frac{t^{N - 2}}{r^{N - 2}} \int_{\partial B_r}  |\nabla f^{(r)}|(x)\, \d\Per(B_r) \, \d t \label{tangential derivative identity}\\
&    =  \frac{r}{N - 1} \int_{\partial B_r} |\nabla f^{(r)}|(x) \, \dPer(B_r)\, . \nonumber 
\end{align}
Let us point out that the latter expression can be viewed as the integral on $\partial B_r$ of the analog of the tangential derivative of $f$ in the smooth case, while $h$ is the radial extension of the values of $f$ on $\partial B_r$ to the whole of $C(X)$.
Given $r>0$, let us introduce the quantity
$$
J(r):= \int_{B_r} |\nabla f| (t,x)\, \d\m_C= \int_0^r t^{N - 1} \int_X |\nabla f| (t,x)\, \d\m \, \d t\, ,
$$
which will approximate $r^{N - 1} \Phi (r)$. Notice that $J$ is a Lipschitz function, hence it is almost everywhere differentiable. Using the identity \eqref{tangential derivative identity},  we obtain that for a.e. $r$ it holds
\begin{equation} \label{First formula for J}
\begin{aligned}
    J'(r) &= \int_{\partial B_r} |\nabla f| (r,x) \, \dPer(B_r)\\ 
    &= \frac{N - 1}{r} \int_{B_r} |D h| (t,x) \, \d\m_C + \int_{\partial B_r}\left ( |\nabla f| (r,x) - |\nabla f^{(r)}|(x)\right )\, \dPer(B_r)\, .
\end{aligned}
\end{equation}
We notice that $\tr_{\partial B_r} h = \tr_{\partial B_r} f$. By defining $\tilde{h}:C(X)\to \R$ to be equal to $h$ inside $B_r$ and $\chi_E$ outside, we observe that \eqref{tangential derivative identity} and \eqref{First formula for J} still hold if we replace $h$ by $\tilde{h}$ (here  it is key that $B_r$ is the open ball). Therefore, in step $1$ we can choose $g=\tilde{h}$ and  \eqref{almost minimizer for its perimeter} reads as 
\begin{equation}\label{eq:JrInttildeh}
    \int_{B_r} |D \tilde{h}|(t,x) \, \d\m_C + a(r) + b(r) \geq J(r)\, .
\end{equation}
Substituting \eqref{eq:JrInttildeh} into \eqref{First formula for J} and rearranging, yields
\begin{align}\label{second formula for J}
    J'(r) - \frac{N - 1}{r}J(r) \geq \int_{\partial B_r}\left ( |\nabla f| (r,x) - |\nabla f^{(r)}|(x)\right )\, \dPer(B_r) - \frac{N - 1}{r}(a(r)+b(r))\, .
\end{align}

With a slight abuse, in order to keep notation simple, below we will write $\nabla\sfd_C(O, (r,x))$ to denote $\nabla(\sfd_C (O, \cdot))(r,x)$.
Using Proposition \ref{characterization of scalar product} together with the fact that $1-\sqrt{1-s} \geq \frac{s}{2}$, for $0\leq s\leq 1$, 
the $\mathrm{BL}$ characterization of the norm \cite[Def.\;3.8]{Gigli2018SobolevSO}  of  $|\nabla f|$ and the identification between minimal weak upper gradients for different exponents on $\RCD$ spaces from \cite{GIGLIIndependenceOnP}, we have
\begin{equation}\label{difference to product}
    \begin{aligned}
    \frac{|\nabla f|(r,x) - |\nabla f^{(r)}|(x)}{|\nabla f|(r,x)}&=1-\sqrt{1-\frac{\left(\nabla f (r,x) \cdot \nabla  \sfd_C (O, (r,x))\right )^2}{|\nabla f|(r,x)^2}}  \\
    &\geq \frac{\left( \nabla f (r,x) \cdot \nabla  \sfd_C (O, (r,x))\right )^2}{2|\nabla f|(r,x)^2}\\
    &=  \frac{\left( \nabla f (r,x) \cdot \nabla (\frac{1}{2}\sfd^2_C (O, (r,x)))\right )^2}{2r^2\left(|\nabla f|(r,x)\right)^2}\, ,
\end{aligned}
\end{equation}
for $\m$-a.e. $x\in X$ and a.e. $r\in(0,+\infty)$. Above, we understand that all the term vanish on the set where $|\nabla f|=0$.

Let us now define the function 
$$\Phi_f (r) := \frac{\int_{B_r} |\nabla f|(t,x) \, \d\m_C}{r^{N - 1}}= \frac{J(r)}{r^{N - 1}}  $$ 
which will approximate the function $\Phi$ in the statement of the theorem. Notice that $\Phi_f$ is Lipschitz and differentiable almost everywhere by the coarea formula \eqref{coarea formula 1}. Taking its derivative and using  \eqref{second formula for J} and \eqref{difference to product} we obtain that for a.e. $r$ it holds
\begin{align} \label{Derivative of approximant}
\Phi_f' (r) &= \frac{J'(r) - \frac{N - 1}{r}J(r)}{r^{N - 1}} \nonumber\\
&\geq \int_{\partial B_r} \frac{\left(\nabla f (r,x) \cdot \nabla \left(\frac{1}{2} \sfd^2_C (O, (r,x))\right)\right )^2}{2r^{N+1}|\nabla f|(r,x)}\, \dPer(B_r) -\frac{N - 1}{r^{N}}(a(r)+ b(r))\, .
\end{align}
Integrating \eqref{Derivative of approximant} from $0< r_1 <r_2 <\infty$, and using coarea formula, we get
\begin{equation}\label{equation: final smooth estimate for the monotonicity formula}
    \begin{aligned}
         \Phi_f(r_2) - \Phi_f (r_1) &
         \geq \int_{B_{r_2}\setminus \overline{B_{r_1}}}\frac{\left(\nabla f (r,x) \cdot \nabla \left(\frac{1}{2} \sfd^2_C (O, (r,x))\right)\right )^2}{2r^{N+1}|\nabla f|(r,x)}\,\d\m_C\\ 
         &\qquad - \int_{r_1}^{r_2}\frac{N - 1}{r^{N}}(a(r)+b(r)) \, \d r\, .
    \end{aligned}
\end{equation}

\textbf{Step 3. Approximation.}\\
In this step we carry out an approximating argument, using step 2 and in particular \eqref{equation: final smooth estimate for the monotonicity formula}. This allows us to conclude the monotonicity part of the theorem.

Let $\{f_k\}_{k \in \N} \subset \mathrm{LIP} (C(X))\cap {\rm D}_{\rm loc}(\Delta)$ be a sequence of non-negative functions converging in $\mathrm{BV_\mathrm{loc}}(X)$ to $\chi_E$. That is, 
\begin{equation}\label{BV convergence of approximants}
    |f_k - \chi_E|_{L^1(B_r)} \stackrel{k \to \infty}{\longrightarrow} 0, \qquad |D f_k | (B_r) \stackrel{k \to \infty}{\longrightarrow} |D \chi_E | (B_r), \qquad \text{for all } r>0\, .  
\end{equation}
Such sequence can be easily constructed by approximation via the heat flow, see for instance \cite{AmbrosioBakryEmery} for analogous arguments.

Let us start by showing that the errors defined in \eqref{errors} relative to $f_k$ go to zero as $k$ tends to $\infty$. 

The term $a_k(r) := \left| |D f_k|(B_r) - \mathrm{Per}(E;B_r) \right| \stackrel{k \to \infty}{\longrightarrow} 0 $ by $\BV_\mathrm{loc}$-convergence of $f_k$ to $\chi_E$, i.e. \eqref{BV convergence of approximants}. 

To deal with the error term $b_k(r) := \int_{\partial B_r} |\tr_{\partial B_r}^{\rm ext} \chi_E - \tr_{\partial B_r} f_k| \dPer{(B_r)} $, we can use the coarea formula \eqref{coarea formula 1} to show that 
$$
    \int_{B_r} |f_k - \chi_E| \, \d\m_C = \int_0^r b_k(s) \, \d s\, .
$$
Together with the $L^1$-convergence of $f_k$ to $\chi_E$, this shows that $b_k(r) \to 0 $ for $\mathcal{L}^1$-a.e.\;$r>0$.  

Lastly, let us show that $\Phi_f(r)\to \Phi(r)$ for $\mathcal{L}^1$-a.e.\;$r>0$. 
By $\BV_\mathrm{loc}$ convergence of $f_k$ to $\chi_E$ \eqref{BV convergence of approximants}, we have
\begin{equation}\label{approximation of monotonicity function}
\begin{aligned}
    \lim_{k\to \infty}\Phi_{f_k}(r)&=  \frac{1}{r^{N - 1}} \lim_{k\to \infty} \int_{B_r} |D f_k| \d\m_C = \frac{1}{r^{N - 1}}\int_{B_r} \dPer(E)= \Phi(r)\, .
\end{aligned}
\end{equation}
Consequently, letting $k\to \infty$ in the estimate \eqref{equation: final smooth estimate for the monotonicity formula} with $f$ replaced by $f_k$,  we obtain
\begin{equation}\label{monotonicity}
    \Phi(r_2) - \Phi(r_1) \geq 0\,, \qquad \mbox{for }\mathcal{L}^1 \mbox{-almost every } \, r_2 > r_1 >0\, ,
\end{equation}
thanks to the non-negativity of the term 
\begin{equation}
    \label{rigidity term}
    \int_{B_{r_2}\setminus \overline{B_{r_1}}}\frac{\left(\nabla f_k (r,x) \cdot \nabla \left(\frac{1}{2} \sfd^2_C (O, (r,x))\right)\right )^2}{2r^{N+1}|\nabla f_k|(r,x)}\,\d\m_C \geq 0\, .
\end{equation}
To conclude that $\Phi$ is monotone, we need to extend \eqref{monotonicity} to every $r_2>r_1 > 0$.

Let $\{r_k\}_{k \in \N}$ be any sequence such that $r_k \uparrow r$. Since $B_r$ is open, $B_{r_k} \uparrow B_r$. Hence, by the inner regularity of measures
$$
    \mathrm{Per}(E; B_{r_k}) \to \mathrm{Per}(E;B_r)\, .
$$
Let $r_2$, $r_1 > 0$. Since the set of radii for which \eqref{monotonicity} holds is dense, we can find $\{r_{1,k}\}_{k \in \N}$ and $\{r_{2,l}\}_{l \in \N}$ for which \eqref{monotonicity} holds and such that $r_{1,k} \uparrow r_1$ and $r_{2,l} \uparrow r_2$. Then
$$
    0 \leq \lim_{l \to \infty} \Phi(r_{2,l}) - \lim_{k \to \infty} \Phi(r_{1,k}) = \Phi(r_2) - \Phi(r_1)\, .
$$
\smallskip

\textbf{Step 4. Rigidity.}\\
In this step we focus on the rigidity part of the statement. We show that if there exist $r_2>r_1>0$ such that $\Phi(r_1)=\Phi(r_2)$, then $E\cap \big(B_{r_2}\setminus \overline{B_{r_1}}\big)$ is a cone.

The first step is to prove the following claim: 
\begin{equation}
    \label{convergence of term}
    \begin{aligned}
     \liminf_{k \to \infty} & \int_{B_{r_2}\setminus \overline{B_{r_1}}}\frac{\left(\nabla f_k (r,x) \cdot \nabla \left(\frac{1}{2} \sfd^2_C (O, (r,x))\right)\right )^2}{2r^{N+1}|\nabla f_k|(r,x)}\,\d\m_C \\
    &\geq \int_{B_{r_2}\setminus \overline{B_{r_1}}} \frac{\left(\nu_E (r,x) \cdot \nabla \left(\frac{1}{2} \sfd^2_C (O, (r,x))\right)\right )^2}{2r^{N+1}}\,\dPer(E)\, ,
    \end{aligned}
\end{equation}
where $\nu_E$ is the unit normal to $E$, see Theorem \ref{thm:NuEGG}. Subsequently, we will be able to conclude using the characterization of cones provided by Lemma \ref{characterization of cones}. 
\smallskip

The plan is to apply Lemma \ref{Joint Lower Semicontinuity} to prove \eqref{convergence of term}.
Using the notation in Lemma \ref{Joint Lower Semicontinuity}, we define the measures 
\begin{equation}
    \label{measure mu k}
    \mu_k := |\nabla f_k| \m_C \llcorner_{(B_{r_2} \setminus \overline{B_{r_1}})}\, , \qquad  \mu := \mathrm{Per}(E; \cdot) \llcorner_{(B_{r_2} \setminus \overline{B_{r_1}})}\, .
\end{equation}
The $\BV_\mathrm{loc}$-convergence of $f_k$ to $\chi_E$ (see \eqref{BV convergence of approximants}) ensures that
$\mu_k \weakto \mu$ in duality with $\C_{\rm b}(C(X))$.
The functions 
\begin{equation*}
    g_k : = \frac{\nabla f_k \cdot \nabla \left(\frac{1}{2} \sfd^2_C (O, \cdot))\right)}{\sqrt{2}r^{\frac{N+1}{2}}|\nabla f_k|} \cdot \chi_{\{|\nabla f_k| >0\}} \in L^2(C(X); \mu_k)
\end{equation*}
satisfy \eqref{equi boundedness lemma}. Indeed, 
$$
\nabla f_k (r,x)\cdot \nabla \left(\frac{1}{2} \sfd^2_C (O, (r,x))\right) \leq \frac{1}{2} |\nabla f_k|(r,x) |\nabla \sfd^2_C(O, (r,x))| = r |\nabla f_k| (r,x)\, \quad\text{$\m_{C}$-a.e.}\, .
$$
Therefore, using \eqref{distance squared to distance},
\begin{equation*}
    \begin{aligned}
            \|g_k\|^2_{(X; \mu_k)} &= \int_{B_{r_2} \setminus \overline{B_{r_1}}} \frac{\left(\nabla f_k \cdot \nabla \left (\frac{1}{2} \sfd^2_C (O, \cdot))\right)\right)^2}{2r^{N+1}|\nabla f_k|} \cdot \chi_{\{|\nabla f_k| >0\}} \d\m_C \\
             & \leq  \int_{B_{r_2} \setminus \overline{B_{r_1}}} \frac{1}{2r^{N-1}} |\nabla f_k|\, \d\m_C < C < +\infty\, ,
    \end{aligned}
\end{equation*}
for some $C>0$ independent of $k \in \N$ thanks to the $\BV_\mathrm{loc}$-convergence \eqref{BV convergence of approximants}.\\ 
Consequently, Lemma \ref{Joint Lower Semicontinuity} provides the existence of $g \in L^2(C(X); \mu)$ and a subsequence $k(l)$ such that 
\begin{equation}\label{unknown weak convergence}
    \frac{\nabla f_{k(l)} \cdot \nabla \left(\frac{1}{2} \sfd^2_C (O, \cdot)\right)}{\sqrt{2}r^{\frac{N+1}{2}}|\nabla f_{k(l)}|} \cdot \chi_{\{|\nabla f_{k(l)}| >0\}}\,\mu_{k(l)} \weakto g\;  \mathrm{Per}(E; \cdot) \llcorner_{(B_{r_2} \setminus \overline{B_{r_1}})}\, ,
\end{equation}
in duality with $\C_{\rm b}(C(X))$.
Up to relabelling the approximating sequence $f_k$, we can suppose that the whole sequence satisfies \eqref{unknown weak convergence}. We next determine the limit function $g$. 

Fix a test function $\varphi \in \mathrm{LIP}\cap {\rm D}(\Delta) (B_{r_2} \setminus \overline{B_{r_1}})$ with compact support contained in $B_{r_2} \setminus \overline{B_{r_1}}$.

We apply the Gauss-Green formula (Theorem \ref{thm:NuEGG}) and use that $\varphi$ has compact support in $B_{r_2} \setminus \overline{B_{r_1}}$ to obtain 
\begin{equation}\label{application of Gauss Green to the approximating sequence}
\begin{aligned}
    \int_{B_{r_2} \setminus \overline{B_{r_1}}} &\;  f_k\, \div \left (\frac{\varphi}{\sqrt{2}r^{\frac{N+1}{2}}} \nabla \left( \frac{1}{2} \sfd^2_C(O,\cdot)\right)\right) \, \d\m_C \\
    &=  - \int_{B_{r_2} \setminus \overline{B_{r_1}}}  \frac{\varphi}{\sqrt{2}r^{\frac{N+1}{2}}} \left( \nabla f_{k} \cdot \nabla \left(\frac{1}{2} \sfd^2_C (O, \cdot)\right) \right) \, \d\m_C \\
    &= - \int_{B_{r_2} \setminus \overline{B_{r_1}}}  \varphi \frac{\nabla f_{k} \cdot \nabla \left(\frac{1}{2} \sfd^2_C (O, \cdot)\right)}{\sqrt{2}r^{\frac{N+1}{2}}|\nabla f_{k}|} \d\mu_{k}\, .
    \end{aligned}
\end{equation}
Using the $L^1$-convergence of $f_k$ to $\chi_E$ and that 
$$
\left \| \frac{\div (\varphi \nabla (\frac{1}{2}\sfd^2_C(O,\cdot))}{\sqrt{2}r^{\frac{N+1}{2}}}\right \|_{L^\infty (B_{r_2} \setminus \overline{B_{r_1}})} < \infty\, ,
$$
we infer that
\begin{align}
    \lim_{k\to \infty} &\int_{B_{r_2} \setminus \overline{B_{r_1}}} f_k \, \div \left (\frac{\varphi}{\sqrt{2}r^{\frac{N+1}{2}}} \nabla \left( \frac{1}{2}\sfd^2_C(O,\cdot) \right)\right)\, \d\m_C = \int_{E} \div \left (\frac{\varphi}{\sqrt{2}r^{\frac{N+1}{2}}} \nabla \left(\frac{1}{2}\sfd^2_C(O,\cdot) \right)\right) \, \d\m_C \nonumber \\
    & = - \int_{\partial^* E} \frac{\varphi} {\sqrt{2}r^{\frac{N+1}{2}}} \; \nabla \left(\frac{1}{2} \sfd^2_C (O, \cdot) \right) \cdot \nu_E \, \dPer(E)\, , \label{approximation of first term coming from Gauss Green}
    \end{align}
where in the last equality we used the Gauss-Green formula (Theorem \ref{thm:NuEGG}). Combining \eqref{application of Gauss Green to the approximating sequence} and \eqref{approximation of first term coming from Gauss Green} we obtain
\begin{equation}
    \label{uniqueness of weak convergence of measures Gauss-Green}
    \begin{aligned}
    \lim_{k\to\infty} \int_{C(X)}  \varphi \frac{\nabla f_{k} \cdot \nabla \left(\frac{1}{2} \sfd^2_C (O, \cdot)\right)}{\sqrt{2}r^{\frac{N+1}{2}}|\nabla f_{k}|} \d\mu_{k} 
    = \int_{\partial^* E} \frac{\varphi} {\sqrt{2}r^{\frac{N+1}{2}}} \nabla \left(\frac{1}{2} \sfd^2_C (O, (\cdot))\right) \cdot \nu_E \, \dPer(E).
    \end{aligned}
\end{equation}
That is, 
\begin{equation}
    \label{gauss green convergence}
     \frac{\nabla f_k \cdot  \nabla \left(\frac{1}{2} \sfd^2_C (O, \cdot))\right)}{{\sqrt{2}r^{\frac{N+1}{2}}}|\nabla f_k|} \mu_k \weakto \frac   {\nabla \left(\frac{1}{2} \sfd^2_C (O, (\cdot))\right) \cdot \nu_E} {{\sqrt{2}r^{\frac{N+1}{2}}}}  \mathrm{Per}(E)
\end{equation}
in duality with $ \C_c(B_{r_2}\setminus \overline{B_{r_1}})$, by approximation. By the uniqueness of the weak limit and from \eqref{unknown weak convergence} we can conclude that
$$
    g =  \frac{\nabla \left(\frac{1}{2} \sfd^2_C (O, \cdot)\right)\cdot \nu_E}{\sqrt{2}r^{\frac{N+1}{2}}} \qquad \mathrm{Per}(E)|_{B_{r_2} \setminus \overline{B_{r_1}}}\mbox{-almost everywhere}\, .
$$
From \eqref{lower semicontinuity part of the lemma} in Lemma \ref{Joint Lower Semicontinuity}, we have
$$
    \liminf_{k \to \infty} \|g_{k}\|^2_{L^2(C(X);\mu_{k})} \geq \|g\|^2_{L^2(C(X); \mu)}\, .
$$
That is, we have shown the claim \eqref{convergence of term}. 
\medskip

We are now in position to improve the estimate \eqref{monotonicity} and use it to show the rigidity part of the theorem. By taking the inferior limit in \eqref{equation: final smooth estimate for the monotonicity formula}, recalling \eqref{approximation of monotonicity function} and that the error terms go to zero from step 3, we use \eqref{convergence of term} to infer
\begin{equation}
    \label{improved monotonicity}
    \Phi(r_2) - \Phi (r_1) \geq \int_{B_{r_2}\setminus \overline{B_{r_1}}} \frac{\left(\nu_E (r,x) \cdot \nabla \left(\frac{1}{2} \sfd^2_C (O, (r,x))\right)\right )^2}{2r^{N+1}}\,\dPer(E) \geq 0,
\end{equation}
for every $r_2>r_1>0$. Since we are assuming $\Phi(r_1)=\Phi(r_2)$, if follows that
$$
    \nabla ( \sfd_C (O, \cdot)) \cdot \nu_{E} =0 \quad \mathrm{Per}(E) \mbox{-a.e.\;on } B_{r_2}\setminus \overline{B_{r_1}}\, .
$$
By applying Lemma \ref{characterization of cones}, we can conclude that $E\cap \big(B_{r_2}\setminus \overline{B_{r_1}}\big)$ is a conical annulus.

\end{proof}

Let us now prove a useful characterization of conical annuli contained in cones over $\mathrm{RCD}$ spaces. The characterization is  based on the properties of the normal to the boundary of the subset: roughly the subset is conical if and only if its normal is orthogonal to the gradient of the distance function from the tip of the ambient conical space. In case the ambient space is Euclidean, the result is classical (see for instance \cite[Proposition 28.8]{maggi_2012}).

\begin{lemma}[Characterization of conical annuli]\label{characterization of cones}

Let $(X,\sfd,\m)$ be an $\mathrm{RCD}(N-2,N-1)$ space and let $C(X)$ be the cone over $X$. Let $E \subset C(X)$ be a locally finite perimeter set and let $0<r_1<r_2<\infty$. Then the measure theoretic interior $E^{(1)}\cap \big(B_{r_2}(O)\setminus \overline{B_{r_1}(O)}\big)$ is a conical annulus if and only if
\begin{equation}\label{cone characterization}
\nabla ( \sfd_C (O, \cdot)) \cdot \nu_{E} =0 \quad \mathrm{Per}(E)\text{-a.e. on }  B_{r_2}(O)\setminus \overline{B_{r_1}(O)}.
\end{equation}
\end{lemma}

\begin{proof}
As in the proof of Theorem \ref{Thm:MonotonicityBody}, to keep notation short we will write $B_r$ to denote the open ball of radius $r>0$ and centered at the tip of the cone, i.e. $B_r=B_r(O)$. Also, we will write  $B_r^X(x)$ for the open ball in $X$, of center $x$ and radius $r>0$.
For simplicity of presentation we will show the equivalence only in the case $r_1=0,\; r_2=\infty$. The general case requires minor modifications.
Moreover, in order to simplify the notation, we assume without loss of generality that $E=E^{(1)}$, as the condition \eqref{cone characterization} is clearly independent of the chosen representative.
\medskip

\textbf{Step 1.}\\ 
We start with some preliminary computations aimed to establish the identity \eqref{derivative of u} below, which will be key in showing the characterization of conical annuli in $C(X)$. 

Using the Gauss-Green and the coarea formulas, we will express the derivative of the function
$$
    u(s):= \m_C (E\cap C(B^X_r(x))\cap B_s)
$$
(suitably rescaled) as the product between the unit normal of $E$ and the gradient of the distance function from the tip of $C(X)$.
 
Let $x \in X$, $r,s >0$. By Lemma \ref{locally finite perimeter of a cone over a set} and Theorem \ref{Thm:CutAndPastePer} the set $F:=E\cap C(B^X_r(x))\cap B_s$ is a set of finite perimeter with 
\begin{align*}
    D \chi_F =& D \chi_E \llcorner_{C(B^X_r(x))\cap B_s} + \nabla( \sfd_C (O, \cdot)) \mathrm{Per}(B_s)\llcorner_{E \cap C(B^X_r(x))} \\
   &+ \nu_{C(B^X_r(x))} \mathrm{Per}(C(B^X_r(x)))\llcorner_{E\cap B_s}\, .
\end{align*}
Using the Gauss-Green formula (Theorem \ref{thm:NuEGG}), the equality for the laplacian of the distance function from the tip on cones \cite[Prop.\;3.7]{ConeMetric},  and cut and paste of sets of locally finite perimeter (Theorem \ref{Thm:CutAndPastePer}), we obtain 
\begin{equation}
\begin{aligned}\label{first computation of measure of F}
   N\cdot u(s) &= \int_F \Delta \left(\frac{1}{2}\sfd^2(O,\cdot)\right) \, \d\m_C   \\
   &= \int_{C(B^X_r(x))\cap B_s \cap \partial^* E}\nabla\left(\frac{1}{2} \sfd^2(O,\cdot)\right)\cdot \nu_E \, \dPer (E) \\
   &\quad + \int_{E \cap C(B^X_r(x))\cap \partial B_s} \nabla\left(\frac{1}{2} \sfd^2(O,\cdot)\right)\cdot \nu_{\partial B_s} \, \dPer (B_s)\\
   &\quad + \int_{E \cap B_s \cap \partial C(B^X_r(x))} \nabla \left(\frac{1}{2} \sfd^2(O,\cdot)\right)\cdot \nu_{C(B^X_r(x))} \, \dPer (C(B^X_r(x)))\, .
   \end{aligned}
\end{equation}
We now study separately the three integrals on the right hand side of  \eqref{first computation of measure of F}, starting from the last one.\\  
Fix a function $\varphi \in \mathrm{LIP}(C(X)) \cap {\rm D}(\Delta)$ with compact support. By applying the Gauss-Green Theorem \ref{thm:NuEGG} on the set of locally finite perimeter $C(B^X_r(x))$, we obtain
\begin{equation*}
\begin{aligned}
&\int_{\partial C(B^X_r(x))} \varphi \, \nu_{C(B^X_r(x))} \cdot \nabla \left (\frac{1}{2} \sfd^2 (O,\cdot) \right) \, \dPer (C(B^X_r(x))) \\
&\quad = - \int_{C(B^X_r(x))} \nabla \varphi \cdot \nabla \left (\frac{1}{2} \sfd^2 (O,\cdot) \right) \d\m_C + \int_{C(B^X_r(x))} \varphi \Delta \left (\frac{1}{2} \sfd^2 (O,\cdot) \right) \d\m_C \\
& \quad = \int_{B^X_r(x)} \int_0^\infty \partial_r \varphi^{(y)}(r) \cdot \partial_r \left (\frac{1}{2} \sfd_{(y)}^2 (r) \right)r^{N - 1} \, \d r \ \d\m(y) + \int_{C(B^X_r(x))} \varphi N \d\m_C \\
& \quad =  - \int_{B^X_r(x)} \int_0^\infty \varphi^{(y)}(r) N r^{N - 1} \, \d r \ \d\m(y) + \int_{C(B^X_r(x))} \varphi N \d\m_C =0,
\end{aligned}
\end{equation*}
where we have used \eqref{radial scalar product}, the definition of $\m_C$ and integration by parts on $\R$.  Since $\varphi$ was arbitrary, we  infer that
$$
    \nu_{C(B^X_r(x))} \cdot \nabla \left (\frac{1}{2} \sfd^2 (O,\cdot) \right)=0  \quad \mathrm{Per} (C(B^X_r(x)))  \text{-a.e.}\, 
$$
and thus
\begin{equation}\label{eq:3rdInt=0}
    \int_{E \cap B_s \cap \partial C(B^X_r(x))} \nabla \left(\frac{1}{2} \sfd^2(O,\cdot)\right)\cdot \nu_{C(B^X_r(x))} \, \dPer (C(B^X_r(x)))=0\, .
\end{equation}
Let us now deal with the second integral appearing in the right hand side of \eqref{first computation of measure of F}.  By the chain rule, we have $\nabla \left (\frac{1}{2} \sfd^2 (O, q) \right) = \sfd(O, q) \nabla \sfd(O,q)$. Therefore, we obtain
\begin{equation}   \label{term on the boundary of the sphere}
\begin{aligned}
    & \int_{E \cap C(B^X_r(x))\cap \partial B_s} \nabla \left(\frac{1}{2} \sfd^2(O,\cdot)\right)\cdot \nu_{\partial B_s} \, \dPer (B_s) \\
    &\quad = \int_{E \cap C(B^X_r(x))\cap \partial B_s} \sfd(O,\cdot)\;  \nabla  \sfd(O,\cdot) \cdot \nu_{\partial B_s} \, \dPer (B_s) \\
    & \quad  = s \mathrm{Per} (B_s; E \cap C(B^X_r(x)))\, ,
\end{aligned}
\end{equation}
where we used \cite[Prop.\;6.1]{SemolaCutAndPaste}. Inserting \eqref{eq:3rdInt=0} and \eqref{term on the boundary of the sphere} into \eqref{first computation of measure of F}, yields
\begin{equation}\label{second computation of measure of F}
\begin{aligned}
    u(s) &= \frac{1}{N}\int_{C(B^X_r(x))\cap B_s}\nabla \left(\frac{1}{2}\sfd^2(O,\cdot)\right)\cdot \nu_E \, \dPer (E)+ \frac{s}{N} \mathrm{Per} (B_s; E \cap C(B^X_r(x)))\, .
\end{aligned}
\end{equation}
By the coarea formula \eqref{coarea formula 1}, $u$ is Lipschitz and differentiable almost everywhere  and it holds 
\begin{equation}\label{coarea u}
    u(s)= \int_0^s \mathrm{Per} (B_t; E \cap C(B^X_r(x)))\, \d t\, .
\end{equation}
We now compute the derivative of $\frac{u(s)}{s^{N}}$. Combining \eqref{second computation of measure of F} and  \eqref{coarea u}, we obtain that for a.e. $s$ it holds
\begin{equation}\label{derivative of u}
    \begin{aligned}
         \frac{\d}{\d s}\frac{u(s)}{s^{N}} = \frac{u'(s)}{s^{N}} - N \frac{u(s)}{s^{N+1}} =  - \frac{\int_{C(B^X_r(x))\cap B_s}\nabla \left(\frac{1}{2}\sfd^2(O,\cdot)\right)\cdot \nu_E \, \dPer (E) }{s^{N+1}}.
    \end{aligned}
\end{equation}
Fix $0<r<\min(s, 1)$. In the following steps, we consider the sets
\begin{equation}\label{cones intersected annuli type sets}
A((s,x), r):= C(B^X_r(x))\cap B_{s(1+r)}\setminus B_{s(1-r)}\, .
\end{equation} 

By Lemma \ref{ratio between radii in A} below, the family of sets $\{A(q,r):\,  q\in C(X), \, r>0\}$ generates the Borel $\sigma$-algebra of $C(X)$, since for any $q \in C(X)$, $r>0$ there exist $r_a$, $r_b>0$ such that 
\begin{equation}\label{equivalence of family of sets A}
B(q,r_a) \subset A(q,r) \subset B(q, r_b)\, .
\end{equation}

\textbf{Step 2.}\\ 
In this step we show that if $E$ is a cone, then  \eqref{cone characterization} holds with $r_1=0, r_2=\infty$. We will first show  that $\frac{u(s)}{s^N}$ is constant, and then conclude using \eqref{derivative of u}.

Since $E$ is a cone, there exists a set $F \subset X$ such that  $E = \{(t,x) \in C(X): x \in F,\ t \geq 0 \}.$ 
Note that $E \cap C(B^X_r(x))$ is a cone for any $x \in X$, $r>0$. Thus, for any $s >0$, it holds:
\begin{align*}
    \m_C(E \cap C(B^X_r(x))\cap B_s ) = \m(F \cap B^X_r(x))\int_{0}^s \rho^{N-1}\, \d\rho  = \frac{s^N}{N} \m(F \cap B^X_r(x))\, ,
\end{align*}
yielding that $s\mapsto s^{-N} u(s)$ is constant.\\ 
From \eqref{derivative of u} it follows that for all $r>0$, $p \in C(X)$ we have
\begin{equation}\label{vanishing average on sets A}
    \int_{A(p, r)}\nabla\left(\frac{1}{2}\sfd^2(O,\cdot)\right)\cdot \nu_E \, \dPer (E) = 0.
\end{equation}
By the Lebesgue differentiation Theorem (see for instance \cite[Rem.\;2.19]{rankOneTheorem}),  for $\mathrm{Per}(E)$-a.e.\;$p\in C(X)$ it holds
\begin{align*}
   \lim_{r \to 0} \dashint_{B_r(p)}\left|\nabla \left(\frac{1}{2}\sfd^2(O,q)\right)\cdot \nu_E(q)  - \nabla \left(\frac{1}{2}\sfd^2(O,p)\right)\cdot \nu_E(p)\right| \dPer (E)(q)= 0\, .
\end{align*}

From \eqref{equivalence of family of sets A} and the asymptotic doubling property of the perimeter we infer 
\begin{align*}
   & \lim_{r \to 0} \frac{1}{\Per(E,A(p,r))} \int_{A(p,r)}\left|\nabla\left(\frac{1}{2}\sfd^2(O,q)\right)\cdot \nu_E(q)  - \nabla \left(\frac{1}{2}\sfd^2(O,p)\right)\cdot \nu_E(p)\right| \dPer (E)(q) \\
   &\quad \le C \lim_{r \to 0} \frac{1}{\Per(E,B_{r_b}(p))} \int_{B_{r_b}(p)}\left|\nabla \left(\frac{1}{2}\sfd^2(O,q)\right)\cdot \nu_E(q)  - \nabla\left(\frac{1}{2}\sfd^2(O,p)\right)\cdot \nu_E(p)\right| \dPer (E)(q)\\
   &\quad = 0\, ,
\end{align*}
where $r_a$ and $r_b$ are as in Lemma \ref{ratio between radii in A}. We can now conclude recalling \eqref{vanishing average on sets A}:
\begin{align*}
    0& = \lim_{r \to 0} \frac{1}{\Per(E,A(p,r))} \int_{A(p,r)}\nabla \left(\frac{1}{2}\sfd^2(O,\cdot)\right)\cdot \nu_E\, \dPer (E)\\ 
    &= \nabla \left(\frac{1}{2}\sfd^2(O,p)\right)\cdot \nu_E(p)\, , \quad \text{for}\quad \Per(E)\text{-a.e. }p\, .
\end{align*}
\smallskip
 
\textbf{Step 3.}\\ 
In this last step we show that \eqref{cone characterization} for $r_1=0, r_2=\infty$ implies that $E$ is a cone.  We will show that given $(s,x) \in E$ and $\lambda>0$, then $(\lambda s,x) \in E$. 

Combining the assumption \eqref{cone characterization} with \eqref{derivative of u}, we obtain that $s\mapsto u(s) / s^N =\m_C(E\cap B_r^X(x)\cap B_s)/s^N$ is constant. Therefore, for $\lambda>0$
\begin{equation}\label{rescaling of measure}
\begin{aligned}
        \m_C(E\cap A((s,x),r))&= \m_C\left(E\cap B_r^X(x)\cap B_{s(1+r)}\right) - \m_C\left(E\cap B_r^X(x)\cap B_{s(1-r)}\right) \\ &= \frac{\m_C\left(E\cap B_r^X(x)\cap B_{\lambda s(1+r)}\right)}{\lambda^N} - \frac{\m_C\left(E\cap B_r^X(x)\cap B_{\lambda s(1-r)}\right)}{\lambda^N} \\ &= \frac{\m_C(E\cap A((\lambda s,x),r))}{\lambda^{N}}.      
\end{aligned}
\end{equation}
Moreover,
\begin{equation}\label{rescaling of measure2}
    \begin{aligned}
    \m_C (A((\lambda t, x),r)) &= \m(B^X_r(x))\int_{\lambda t (1-r)}^{\lambda t (1+r)} s^{N - 1} \, \d s = \lambda^{N} \m(B^X_r(x))\int_{ t (1-r)}^{ t (1+r)} s^{N - 1} \, ds\\ 
    &= \lambda^{N}\m_C (A((t, x),r))),
    \end{aligned}
\end{equation}
The combination of  \eqref{rescaling of measure} and \eqref{rescaling of measure2} gives 
\begin{equation}\label{rescaling of measure3}
    \frac{\m_C(E\cap A((s,x),r))}{\m_C (A((s,x),r))} =\frac{\m_C(E\cap A((\lambda s,x),r))}{\m_C (A((\lambda s,x),r))}\, .
\end{equation}
We next show that  if $(s,x)\in E$ then $(\lambda s, x) \in E$. Thanks to \eqref{rescaling of measure3}, it is enough to show that $q \in E$ if and only if 
\begin{equation}\label{eq:DensityEA}
    \lim_{r \to 0} \frac{\m_C(E\cap A(q,r))}{\m_C (A(q,r))} =1\, .
\end{equation}
Assume by contradiction that \eqref{eq:DensityEA} holds but $q \not \in E$. Then, 
\begin{align*}
    \liminf_{r\to 0} \frac{\m_C ((X\setminus E)\cap B_r(q))}{\m_C (B_r(q))} \geq \varepsilon >0\, . 
\end{align*}
Using Lemma \ref{ratio between radii in A}, we infer that
\begin{equation}
    \label{contradiction measure theoretic interior}
\begin{aligned}
        &\liminf_{r\to 0} \frac{\m_C((X\setminus E)\cap A(q,r))}{\m_C (A(q,r))} \\
        &\geq  \liminf_{r\to 0} \frac{\m_C ((X\setminus E)\cap B_{r_a(q,r)}(q))}{\m_C (B_{r_a(q,r)}(q))} \cdot \frac{\m_C (B_{r_a(q,r)}(q))}{\m_C (B_{r_b(q,r)}(q))}\, .
\end{aligned}
\end{equation}
Since $C(X)$ is an $\RCD(0,N)$ space, the Bishop-Gromov monotonicity formula \cite{Sturm2} gives
$$
    \m_C (B_{r_a}(q)) \geq \left (\frac{r_a}{r_b}\right)^N  \m_C (B_{r_b}(q))\, .
$$
Therefore, from \eqref{contradiction measure theoretic interior} we may conclude
\begin{equation}\label{eq:liminfmXEA}
\begin{aligned}
    \liminf_{r\to 0} \frac{\m_C((X\setminus E)\cap A(q,r))}{\m_C (A(q,r))} &\geq  \liminf_{r \to 0} \left (\frac{r_a(q,r)}{r_b(q,r)}\right)^N  \cdot \liminf_{r \to 0} \frac{\m_C ((X\setminus E)\cap B_{r_a(q,r)}(q))}{\m_C (B_{r_a(q,r)}(q)))} \\
    &\geq  \, C \varepsilon > 0\, ,
\end{aligned}
\end{equation}
where $C :=  \liminf_{r \to 0} \big (\frac{r_a(q,r)}{r_b(q,r)}\big)^N >0$  thanks to \eqref{limit of the ratio of radii}. Clearly, \eqref{eq:liminfmXEA} contradicts \eqref{eq:DensityEA}. The proof that $q\in E$ implies \eqref{eq:DensityEA} is analogous.
 
\end{proof}

The following technical lemma was used in the proof of Lemma \ref{characterization of cones} above.

\begin{lemma}\label{ratio between radii in A}
Let $N\ge 2$, let $(X,\sfd,\m)$ be an $\mathrm{RCD}(N-2,N-1)$ space and let $(C(X), \sfd_C,\m_C)$ be the cone over it. If $N=2$, assume also that $\mathrm{diam}(X)\le \pi$. For $x\in X$ and $0<r<s$, consider the sets
\begin{equation}\label{cones intersected annuli type sets2}
A((s,x), r):= C(B^X_r(x))\cap B_{s(1+r)}(O)\setminus B_{s(1-r)}(O).
\end{equation}
Then:
\begin{itemize}
\item The family of sets $\{A(q,r):\,  q\in C(X), \, r>0\}$ generates the Borel $\sigma$-algebra of $C(X)$;
\item  For any $q \in C(X)$, $r>0$ there exist $r_a=r_a(q,r)$ and $r_b=r_b(q,r)>0$ such that 
$$
B_{r_a}(q) \subset A(q,r) \subset B_{r_b}(q)
$$ 
and
\begin{equation}
    \label{limit of the ratio of radii}
    \lim_{r \to 0}\frac{r_a}{r_b}=\frac{1}{4\sqrt{2}}.
\end{equation}
\end{itemize}
\end{lemma}

\begin{proof}
The first claim follows from the second one; thus let us determine $r_a$ and $r_b>0$ that satisfy the second statement. To this aim, we compute the minimal and maximal distance of $q = (t,x)$ from the set $\partial A(q,r)$. Let us start from the minimal distance. We deal with the shell part first: given  $(t(1+r),y) \in \partial B_{t(1+r)} \cap \partial A$ there holds, using \eqref{explicit expression for the distance}
\begin{equation}
    \label{minimum distance from shells}
    \begin{aligned}
    \sfd^2_C ((t,x),(t(1+r), y)) &= t^2 + t^2(1+r)^2 -2t^2(1+r) \cos(\sfd(x,y)) \\
    &\geq t^2 + t^2(1+r)^2 -2t^2(1+r) = t^2 r^2\, ,
    \end{aligned}
\end{equation}
where the equality is achieved at $y=x$. Let now $(s,y) \in \partial C(B^X_r(x)) \cap \partial A(q,r)$:
\begin{equation}
    \label{minimum distance from cylindrical part computation}
    \begin{aligned}
    \sfd^2_C ((t,x),(s,y)) = t^2 + s^2 -2 st \cos(r)\, .
    \end{aligned}
\end{equation}
This defines a differentiable function of $s \in [t(1-r),t(1+r)]$. Its derivative $\partial_s \sfd^2_C ((t,x),(s,y)) = 2s -2t\cos(r)$ is increasing and vanishes at $s=t\cos(r)$. Therefore, we have 
\begin{equation}
    \label{inner radii estimate}
    \sfd^2_C(q, \partial A(q,r)) =   t^2 \sin^2(r)  \, .
\end{equation}
Therefore, we may pick
\begin{equation}
    \label{choice of r_a}
    \begin{aligned}
        r_a  = r_a(q,r)  := \frac{1}{2} \sfd_C(q, \partial A(q,r))  = \frac{1}{2} t \sin(r)\, .
    \end{aligned}
\end{equation}
Next, let us compute the maximal distance of $q$ from $\partial A(q,r)$. Since the maximal distance is attained at the intersection of the shell with the side part of $A$ (by monotonicity of both formulas \eqref{minimum distance from shells} and \eqref{minimum distance from cylindrical part computation} with respect to $\sfd(x,y)$ and $s$, respectively), we can simply compute the maximum by looking at the distance from the top shell. We again compute, for $\sfd(x,y) = r$,
\begin{align*}
    \sfd^2_C ((t,x),(t(1+r), y)) &= t^2 + t^2(1+r)^2 -2t^2(1+r) \cos(r) \\
    &= t^2(2 + 2r +r^2 - 2(1+r)\cos(r))\, .
\end{align*}
Consequently, we may pick
\begin{equation}
    \label{choice of r_b}
    r_b = r_b(q,r) := 2 t \sqrt{(2 + 2r +r^2 - 2(1+r)\cos(r))}\, .
\end{equation}
It is easy to check that $r_a, r_b>0$ defined in  \eqref{choice of r_a},  \eqref{choice of r_b} satisfy \eqref{limit of the ratio of radii}.
\end{proof}

A useful technical tool, used to prove the rigidity statement of the monotonicity formula,  is the following lemma (see \cite[Lemma 10.1]{AmbrosioOptimalTransport} for the proof).

\begin{lemma}[Joint Lower Semicontinuity]
\label{Joint Lower Semicontinuity}
Let $(X,\sfd)$ be a Polish space. Let $\mu, \mu_k \in \mathcal{M}_+(X)$ with $\mu_k \weakto \mu$ in duality with $ \C_b(X)$. Let $g_k \subset L^2(X;\mu_k)$ be a sequence of functions such that 
\begin{equation}
    \label{equi boundedness lemma}
    \sup_{k \in \N} \|g_k\|_{L^2(X;\mu_k)} < \infty.
\end{equation}
Then, there exists a function $g \in L^2(X;\mu)$ and a subsequence $k(l)$ such that 
\begin{equation}
    \label{weak convergence part of the lemma}
    g_{k(l)} \, \mu_{k(l)} \weakto g \mu 
\end{equation}
in duality with $ \C_b(X)$ and 
\begin{equation}
    \label{lower semicontinuity part of the lemma}
    \liminf_{l \to \infty} \|g_{k(l)}\|_{L^2\left(X;\mu_{k(l)}\right)} \geq \|g\|_{L^2(X; \mu)}\, .
\end{equation}
\end{lemma}

\section{Stratification of the Singular Set and further applications}\label{section: stratification}

The first goal of this section is to prove sharp Hausdorff dimension estimates for the singular strata of locally perimeter minimizing sets in $\RCD(K,N)$ spaces $(X,\sfd,\mathcal{H}^N)$. The statement is completely analogous to the classical one for singular strata of minimizing currents in the Euclidean setting, see \cite{Federerstrata}, and for the singular strata of non-collapsed Ricci limits \cite{CheegerColdingStructure1} and $\RCD$ spaces \cite{GigliNonCollapsed}. Also the proof is based on the classical  Federer's dimension reduction argument, and builds upon the monotonicity formula and associated rigidity for perimeter minimizing sets in $\RCD(0,N)$ metric measure cones, Theorem \ref{Thm:MonotonicityBody}. Though, a difference between the present work and the aforementioned papers is that the monotonicity formula is available only at the level of blow-ups and not in the space $X$; this creates some challenges that are addressed in the proof. \\
The second main goal will be to present an application of the monotonicity formula and the associated rigidity for cones, to the existence of perimeter minimizing cones in any blow-down of an $\RCD(0,N)$ space $(X,\sfd,\mathcal{H}^N)$ with Euclidean volume growth. 

\medskip

Below we introduce the relevant definition of singular strata and of interior or boundary regularity points for a locally perimeter minimizing set $E\subset X$, when $(X,\sfd,\mathcal{H}^N)$ is an $\RCD(K,N)$ metric measure space.

\begin{definition}[Singular Strata]\label{definition | singular stratum of a finite perimeter set}
Let $(X,\sfd,\mathcal{H}^N)$ be an $\mathrm{RCD}(K,N)$ space, $E\subset X$ a locally perimeter minimizing set  in the sense of Definition \ref{def:LocPerMin} and $0 \leq k \leq N-3$ an integer. The $k$-singular stratum of $E$,  $\mathcal{S}^E_k$, is defined as
\begin{equation}
    \label{singular stratum definition}
    \begin{aligned}
    \mathcal{S}_k^E := &\{x \in \partial E: \mbox{ no tangent space to $(X,\sfd,\mathcal{H}^N,x, E)$ is of the form }(Y,\rho, \mathcal{H}^N,y,F), \\
    &\; \mbox{ with }(Y,\rho,y) \mbox{ isometric to }(Z\times\R^{k+1},\sfd_Z \times \sfd_\mathrm{eucl},(z,0)) \mbox{ for some pointed }(Z,\sfd_Z,z)\\
    &\; \mbox{ and }F=G\times\R^{k+1} \mbox{ with } G\subset Z \mbox{ global perimeter minimizer}\}.
    \end{aligned}
\end{equation}
\end{definition}

The above definition would make sense also in the cases when $k\ge N-2$. However, it seems more appropriate not to adopt the terminology \emph{singular strata} in those instances.

\begin{definition}[Interior and Boundary Regularity Points]\label{def:regpoint}
 Let $(X,\sfd,\mathcal{H}^N)$ be an $\mathrm{RCD}(K,N)$ space and let $E\subset X$ be a locally perimeter minimizing set in the sense of Definition \ref{def:LocPerMin}. Given $x\in \partial E$, we say that $x$ is an \emph{interior regularity} point if
 \begin{equation}
     \mathrm{Tan}_x(X,\sfd,\mathcal{H}^N,E)=\{(\R^N,\sfd_{\mathrm{eucl}},\mathcal{H}^N, 0,\R^N_+)\}\, .
 \end{equation}
 The set of interior regularity points of $E$ will be denoted by $\mathcal{R}^E$.\\
Given $x\in \partial E$, we say that $x$ is a \emph{boundary regularity} point if
  \begin{equation}
     \mathrm{Tan}_x(X,\sfd,\mathcal{H}^N,E)=\{(\R^N_+,\sfd_{\mathrm{eucl}},\mathcal{H}^N, 0,\{x_1\ge 0\})\}\, ,
 \end{equation}
 where $x_1$ is one of the coordinates of the $\R^{N-1}$ factor in $\R^{N}_+=\R^{N-1}\times \{x_N\ge 0\}$. The set of boundary regularity points of $E$ will be denoted by $\mathcal{R}^E_{\partial X}$.
\end{definition}

It was proved in \cite{WeakLaplacian} that the interior regular set $\mathcal{R}^E$ is topologically regular, in the sense that it is contained in a H\"older open manifold of dimension $N-1$.   By a blow-up argument, in the next proposition, we show that $ \mathrm{dim}_\mathcal{H}\mathcal{R}^E_{\partial X}\leq N-2$.

\begin{proposition}\label{prop:RECod2}
Let $(X,\sfd,\mathcal{H}^N)$ be an $\mathrm{RCD}(K,N)$ space. Let $E\subset X$ be a locally perimeter minimizing set and let $\mathcal{R}^E_{\partial X}$ be the set of  boundary regularity points of $E$, in the sense of Definition \ref{def:regpoint}. Then
\begin{equation}
   \mathrm{dim}_\mathcal{H} \mathcal{R}^E_{\partial X} \leq N-2\, .
 \end{equation}
\end{proposition}

\begin{proof}
We argue by contradiction. Assume there exists $k>N-2$, $k\in\R$ such that 
\begin{equation}\label{Contradiction assumption in BoundReg}
    \mathcal{H}^{k} \left(  \mathcal{R}^E_{\partial X} \right)>0\, .
\end{equation}
Let $\varepsilon>0$. We define the quantitative $\varepsilon$-singular set to be
      \begin{equation}  \label{epsilon singular stratum definitionN-1}
    S^{\varepsilon}(E) := \{x \in X: \ \mathcal{D}((B^X_r(x), \sfd, \mathcal{H}^N, x, E), (B^{\R^{N}}_r,  \sfd_\mathrm{eucl}, 0, \R^{N}_+)) \geq \varepsilon r, \mbox{ for all }r\in(0,\varepsilon) \}\, . 
    \end{equation}
    Recall that the distance $\mathcal{D}$ was introduced in \cite[Definition A.3]{1BakryEmeryAmbrosio}.  Notice that $S^{\varepsilon_1}(E) \supset S^{\varepsilon_2}(E)$ for $0<\varepsilon_1 \leq \varepsilon_2$ and that     \begin{equation}\label{eq:deE-RE}
   \partial E \setminus \mathcal{R}^E = \bigcup_{n \in \N} S^{\varepsilon_n}(E),
   \end{equation}
    for any sequence $\varepsilon_n \downarrow 0$. 
 It is also clear that  
  \begin{equation}\label{eq:deE-RE2}
    \mathcal{R}^E_{\partial X} \subset   \partial E \setminus \mathcal{R}^E.
   \end{equation}
   The combination of \eqref{Contradiction assumption in BoundReg},  \eqref{eq:deE-RE} and \eqref{eq:deE-RE2} implies that   there exists $\Bar{\varepsilon}>0$ such that 
\begin{equation}
    \mathcal{H}^{k} \left( S^{\Bar{\varepsilon}} (E) \cap \mathcal{R}^E_{\partial X}\right)>0\, .
\end{equation}
By \cite[Theorem 2.10.17]{Federerbook}, there exists $x \in S^{\Bar{\varepsilon}}(E)\cap \mathcal{R}^E_{\partial X}$ such that
\begin{equation} \label{initial estimate in H infinitySE}
    \limsup_{r \to 0} \frac{\mathcal{H}^{k}_\infty \left(B_r(x) \cap  S^{\Bar{\varepsilon}}(E)\cap \mathcal{R}^E_{\partial X}\right)}{r^{k}} \geq C_{k}>0\, ,
\end{equation}
where we denoted by $\mathcal{H}^{k}_{\infty}$ the $k$-dimensional $\infty$-pre-Hausdorff measure.

By the very definition of $\mathcal{R}^E_{\partial X}$,  for every sequence $r_i \downarrow 0$, $E \subset (X, \sfd/r_i, \mathcal{H}^N/r_i^N, x)$ converges in the sense of Definition \ref{definition | L^1 convergence of sets} to a quadrant $\{x_1\ge 0\}$, where $x_1$ is one of the coordinates of the $\R^{N-1}$ factor in $\R^{N}_+=\R^{N-1}\times \{x_N\ge 0\}$. 

Embedding the sequence of rescaled spaces $X_i$ and their limit $\R^{N}_{+}$ into a proper realization of the pGH-convergence, by Blaschke’s theorem (cf. \cite[Theorem 7.3.8]{burago2001course}) there exist a compact set $A\subset \R^{N}_{+}$ and a subsequence, which we do not relabel, such that $S^{\Bar{\varepsilon}}(E)\cap \mathcal{R}^E_{\partial X}\cap B^i_1(x)$ converges to $A$ in the Hausdorff sense.\\ 
Moreover, it is elementary to check that $A \subset S^{\Bar{\varepsilon}}(\{x_1\ge 0\})$ in $\R^{N}_{+}$. Therefore, we obtain
\begin{equation}
    \label{Hausdorff pre-measure passes to the limitRE}
    \begin{aligned}
    \mathcal{H}^{k}_\infty \left(S^{\Bar{\varepsilon}}(\{x_1\ge 0\})\right) &\geq \mathcal{H}^{k}_\infty \left(A\right)
    \geq \limsup_{i \to \infty}\mathcal{H}^{k}_\infty \left(S^{\Bar{\varepsilon}}(E)\cap \mathcal{R}^E_{\partial X}\cap B^{i}_1(x)\right)\\
    &= \limsup_{i \to \infty} \frac{\mathcal{H}^{k}_\infty \left(B_{r_i}(x) \cap  S^{\Bar{\varepsilon}}(E)\cap \mathcal{R}^E_{\partial X}\right)}{r_i^{k}} >0\, ,
    \end{aligned}
\end{equation}
where we relied on the classical upper semicontinuity of the pre-Hausdorff measure with respect to Hausdorff convergence in the second inequality and on \eqref{initial estimate in H infinitySE} in the last one. However, it is easy to check that $S^{\Bar{\varepsilon}}(\{x_1\ge 0\})= \{x_{1}=x_{N}=0\}$ which has Hausdorff co-dimension 2, contradicting \eqref{Hausdorff pre-measure passes to the limitRE}.
\end{proof}

Our main results about the stratification of the singular set for perimeter minimizers are that the complement of $\mathcal{S}_{N-3}^E$ in $\partial E$ consists of either interior or boundary regularity points, and that the classical Hausdorff dimension estimate $\dim(\mathcal{S}^E_k)\le k$ holds for any $0\le k\le N-3$. Below are the precise statements.

\begin{theorem}\label{thmn: compl n-3}
Let $(X,\sfd,\mathcal{H}^N)$ be an $\mathrm{RCD}(K,N)$ space and let $E\subset X$ be a locally perimeter minimizing set  in the sense of Definition \ref{def:LocPerMin}. Then 
\begin{equation}
    \partial E\setminus \mathcal{S}_{N-3}^E=\mathcal{R}^E\cup \mathcal{R}^E_{\partial X}\, .
\end{equation}
\end{theorem}

\begin{theorem}[Stratification of the singular set]\label{thm: strat core}
Let $(X,\sfd,\mathcal{H}^N)$ be an $\mathrm{RCD}(K,N)$ space and $E\subset X$ a locally perimeter minimizing set. Then, for any $0\le k\le N-3$ it holds
\begin{equation}
    \mathrm{dim}_\mathcal{H} \mathcal{S}_k^E \leq k\, .
\end{equation}
\end{theorem}

Another application of the monotonicity formula with the associated rigidity is that  if an $\RCD(0,N)$ space $(X,\sfd,\mathcal{H}^N)$  with Euclidean volume growth contains a global perimeter minimizer, then any asymptotic cone  contains a  perimeter minimizing cone. 

\begin{theorem}\label{thm:blowdown}
Let $(X,\sfd,\mathcal{H}^N)$ be an $\RCD(0,N)$ metric measure space with Euclidean volume growth, i.e. satisfying for some (and thus for every) $x\in X$:
\begin{equation}\label{eq:AssEVG}
\liminf_{r\to \infty} \frac{\mathcal{H}^N(B_r(x))}{r^N} >0.
\end{equation}
Let $E\subset X$ be a global perimeter minimizer in the sense of Definition  \ref{def:LocPerMin}. Then for any blow-down $(C(Z),\sfd_{C(Z)},\mathcal{H}^N)$ of  $(X,\sfd,\mathcal{H}^N)$ there exists a cone $C(W)\subset C(Z)$ global perimeter minimizer.
\end{theorem}

\begin{remark}
The conclusion of Theorem \ref{thm:blowdown} above seems to be new also in the more classical case of smooth Riemannian manifolds with non-negative sectional curvature, or non-negative Ricci curvature. We refer to \cite{Andersonmini} for earlier progress in the case of smooth manifolds with non-negative sectional curvature satisfying additional conditions on the rate of convergence to the tangent cone at infinity and on the regularity of the cross section and to \cite{DingJostXin} for the case of smooth Riemannian manifolds with non-negative Ricci curvature and quadratic curvature decay.
\end{remark}

\begin{proof}[Proof of Theorem \ref{thmn: compl n-3}]
    Let us consider a point $x\in  \partial E\setminus \mathcal{S}_{N-3}^E$. By the very definition of the singular stratum $\mathcal{S}_{N-3}^E$, there exists a tangent space to $(X,\sfd,\mathcal{H}^N,E)$ at $x$ of the form $(\R^{N-2} \times Z,\sfd_{\mathrm{eucl}}\times \sfd_Z,\mathcal{H}^N,y,G\times \R^{N-2})$, where $(Z,\sfd_Z,\mathcal{H}^2)$ is an $\RCD(0,2)$ metric measure cone (because all tangent cones to an $\RCD(K,N)$ space $(X,\sfd,\mathcal{H}^N)$ are metric measure cones \cite{GigliNonCollapsed}) and $G\subset Z$ is a globally perimeter minimizing set (in the sense of Definition  \ref{def:LocPerMin}) thanks to \cite[Theorem 2.42]{WeakLaplacian}.

    By Lemma \ref{lemma: 2d} there are only two options. Either $x$ is an interior point and a tangent space is $(\R^N,\sfd_{\mathrm{eucl}},\mathcal{H}^N,0,\R^N_+)$, or $x$ is a boundary point and a tangent space is $(\R^N_+,\sfd_{\mathrm{eucl}},\mathcal{H}^N,0,\{x_1\ge 0\})$.\\
    In the first case, it was shown in \cite{WeakLaplacian} that the tangent space at $x$ is unique and hence $x\in \mathcal{R}^E$. If the second possibility occurs, then by \cite{BrueNaberSemola} we infer that the tangent cone to the ambient space $(X,\sfd,\mathcal{H}^N)$ is unique. The uniqueness of the tangent cone to the set of finite perimeter can be obtained with an argument completely analogous to the one used for interior points in \cite{WeakLaplacian}, building on top of the classical boundary regularity theory (cf. for instance with \cite{Gruter}) instead of the classical interior regularity theory for perimeter minimizers in the Euclidean setting. Hence $x\in \mathcal{R}^E_{\partial X}$ is a boundary regularity point.
\end{proof}

\begin{proof}[Proof of Theorem \ref{thm: strat core}]
We argue by contradiction via Federer's dimension reduction argument. The proof is divided into four steps. In the first step we set up the contradiction argument and reduce to the case of entire perimeter minimizers inside $\RCD(0,N)$ metric measure cones. In the second step we make a further reduction to the case when the perimeter minimizer is a cone itself, building on top of Theorem \ref{Thm:MonotonicityBody}. Via additional blow-up arguments we gain a splitting for the ambient space and for the perimeter minimizing set in step three, thus performing a dimension reduction. The argument is completed in step four. A key subtlety with respect to more classical situations is that the monotonicity formula holds only for perimeter minimizers centered at vertices of metric measure cones, resulting into the necessity of iterating the blow-ups. 
\medskip

\textbf{Step 1.} \\
We argue by contradiction. Suppose that the statement does not hold for some $0\le k\le N-3$. Then, there exists $k'>k$, $k'\in\R$ such that 
\begin{equation}\label{Contradiction assumption in the stratification}
    \mathcal{H}^{k'} \left( \mathcal{S}^E_k \right)>0\, .
\end{equation}
Let $\varepsilon>0$. We define the quantitative $(k,\varepsilon)$-singular stratum to be
\begin{equation}
    \label{epsilon singular stratum definition}
    \begin{aligned}
    S_{k,\varepsilon}^{E} := \{x \in X: \ &\mathcal{D}((B^X_r(x), \sfd, \mathcal{H}^N, x, E),(B^{\R^{k+1} \times Z}_r, \sfd_\mathrm{eucl} \times \sfd_{Z}, \mathcal{H}^N,   (0, z), F)) \geq \varepsilon r \\
    & \mbox{ for all }r\in(0,\varepsilon),\ (Z,\sfd_Z,z) \mbox{ pointed spaces and }\\
    & F= \R^{k+1} \times G \mbox{ with } G\subset Z \mbox{ global perimeter minimizer}\}\, . 
    \end{aligned}
\end{equation}
Recall that the distance $\mathcal{D}$ was introduced in \cite[Definition A.3]{1BakryEmeryAmbrosio}. 
Moreover, we notice that $S_{k,\varepsilon_1}^{E} \supset S_{k,\varepsilon_2}^{E}$ for $0<\varepsilon_1 < \varepsilon_2$ and that $S_k^{E} = \bigcup_{n \in \N} S_{k,\varepsilon_n}^{E}$, for any sequence $\varepsilon_n \downarrow 0$. 
\smallskip

The contradiction argument assumption \eqref{Contradiction assumption in the stratification} implies that there exists $\Bar{\varepsilon}>0$ such that 
\begin{equation}
    \mathcal{H}^{k'} \left( S^{E}_{k,\Bar{\varepsilon}} \right)>0\, .
\end{equation}
By \cite[Theorem 2.10.17]{Federerbook}, there exists $x \in S^{E}_{k, \Bar{\varepsilon}}$ such that
\begin{equation}
    \label{initial estimate in H infinity}
    \limsup_{r \to 0} \frac{\mathcal{H}^{k'}_\infty \left(B_r(x) \cap  S^{E}_{k,\Bar{\varepsilon}}\right)}{r^{k'}} \geq C_{k'}>0\, ,
\end{equation}
where we denoted by $\mathcal{H}^{k'}_{\infty}$ the $k'$-dimensional $\infty$-pre-Hausdorff measure.
 
Then there exists a sequence $r_i \downarrow 0$ such that $E \subset (X, \sfd/r_i, \mathcal{H}^N/r_i^N, x)$ converges in the sense of Definition \ref{definition | L^1 convergence of sets} to a global perimeter minimizer (in the sense of Definition  \ref{def:LocPerMin}) $F \subset (C(Z), \sfd_C, \mathcal{H}^N)$. Here we used \cite[Corollary 3.4]{1BakryEmeryAmbrosio} in combination with Lemma \ref{lemma | density of perimeter minimizers} for the compactness, \cite[Theorem 2.42]{WeakLaplacian} for the perimeter minimality of $F$ and \cite{GigliNonCollapsed} to infer that the ambient tangent space is a cone. Here $(Z,\sfd_Z,\mathcal{H}^{N-1})$ is an $\RCD(N-2,N-1)$ metric measure space.

\medskip
Embedding the sequence of rescaled spaces $X_i$ and their limit $C(Z)$ into a proper realization of the pGH-convergence, by Blaschke’s theorem (cf. \cite[Theorem 7.3.8]{burago2001course}) there exist a compact set $A\subset C(Z)$ and a subsequence, which we do not relabel, such that $S^{E}_{k,\Bar{\varepsilon}}\cap B^i_1(x)$ converges to $A$ in the Hausdorff sense.\\ 
Moreover, it is straightforward to check that $A \subset S^{F}_{k,\Bar{\varepsilon}}$. Therefore, we obtain
\begin{equation}
    \label{Hausdorff pre-measure passes to the limit}
    \begin{aligned}
    \mathcal{H}^{k'}_\infty \left(S^{F}_{k, \Bar{\varepsilon}}\right) &\geq \mathcal{H}^{k'}_\infty \left(A\right)
    \geq \limsup_{i \to \infty}\mathcal{H}^{k'}_\infty \left(S^{E}_{k,\Bar{\varepsilon}}\cap B^i_1(x)\right)\\
    &= \limsup_{i \to \infty} \frac{\mathcal{H}^{k'}_\infty \left(B_{r_i}(x) \cap  S^{E}_{k,\Bar{\varepsilon}}\right)}{r_i^{k'}} >0\, ,
    \end{aligned}
\end{equation}
where we relied on the classical upper semicontinuity of the pre-Hausdorff measure with respect to Hausdorff convergence in the second inequality and on \eqref{initial estimate in H infinity} in the last one. 

Lastly, \eqref{Hausdorff pre-measure passes to the limit} implies that 
\begin{equation}\label{eq:measest}
\mathcal{H}^{k'}\left(B_1^{C(Z)} \cap  S^{F}_{k,\Bar{\varepsilon}}\right) >0 \, .
\end{equation}

\medskip

\textbf{Step 2.}\\
In this step, by performing a second blow up, we apply Theorem \ref{Thm:MonotonicityBody} to show that we can also suppose that the global perimeter minimizer is a cone (with respect to a vertex of the ambient cone).  For the sake of clarity, we recall that the set of vertices of $C(Z)$ is the collection of all points $y\in C(Z)$ such that $C(Z)$ is a metric cone centered at $y$. Moreover, we remark that the set of vertices is isometric to $\R^k$ for some $0\le k\le N$.
\smallskip

We claim that there is a point $O\in C(Z)$ such that $O$ is a vertex of $C(Z)$ and the following density estimate holds:
\begin{equation}
    \limsup_{r \to 0} \frac{\mathcal{H}^{k'}_\infty \left(B_r(O) \cap  S^{F}_{k,\Bar{\varepsilon}}\right)}{r^{k'}} \geq C_{k'}>0\, .
\end{equation}
If the claim does not hold, then by \eqref{eq:measest} there are points of density for $\mathcal{H}^{k'}_{\infty}$ restricted to $S^{F}_{k,\Bar{\varepsilon}}$ and none of them belongs to the set of vertices of $C(Z)$. Hence we can repeat the argument in step 1, blowing up at a density point for $\mathcal{H}^{k'}_{\infty}$ restricted to $S^{F}_{k,\Bar{\varepsilon}}$ which is not a vertex in the ambient cone. In this way, the dimension of the set of vertices of the ambient space, which is isometric to a Euclidean space, increases at least by one.\\
The procedure can be iterated until one of the following two possibilities occurs: the ambient is isometric to $\R^N$, with standard structure, in which case \eqref{eq:measest} contradicts the classical regularity theory, or there is a density point for $\mathcal{H}^{k'}_{\infty}$ restricted to $S^{F}_{k,\Bar{\varepsilon}}$ which is also a vertex of $C(Z)$. 
\smallskip

Let now $O$ denote any such vertex of $C(Z)$. By Theorem \ref{Thm:MonotonicityBody} and the density estimates in Lemma \ref{lemma | density of perimeter minimizers}, the map
\begin{equation}
r\mapsto \frac{\mathrm{Per}(F; B_r(O))}{r^{N-1}}
\end{equation}
is monotone non-decreasing, bounded and bounded away from $0$. Therefore, there exists the limit 
\begin{equation}\label{limit in the monotonicity formula first blow up}
    0< a:= \lim_{r\to 0} \frac{\mathrm{Per}(F;B_r(O))}{r^{N-1}} < \infty\, .
\end{equation}

We perform a second blow up at the tip $O\in C(Z)$ and obtain a global perimeter minimizer $G\subset C(Z)$. By \eqref{limit in the monotonicity formula first blow up} and Theorem \ref{Thm:MonotonicityBody}, $G$ is a cone. Moreover, by repeating the arguments in step 1, taking into account that $O$ was chosen to be a density point for $\mathcal{H}^{k'}_{\infty}$ restricted to $S^{F}_{k,\Bar{\varepsilon}}$, it holds 
\begin{equation}
    \label{Hausdorff measure after second blow up}
    \mathcal{H}^{k'}\left(S^{G}_{k,\Bar{\varepsilon}}\right) >0\, .
\end{equation}
It follows from \eqref{Hausdorff measure after second blow up} that there exists a point in $S^{G}_{k, \Bar{\varepsilon}} \setminus \{O\}$.
\medskip

\textbf{Step 3.}\\
The goal of this step is to gain a splitting for the ambient and the perimeter minimizer set by considering a blow-up of $G$ at a density point for $\mathcal{H}^{k'}_{\infty}$ restricted to $S^{G}_{k,\Bar{\varepsilon}}$ that is not a vertex. Roughly speaking, we will achieve this by showing that the unit normal of the blow-up is everywhere perpendicular to the gradient of a splitting function obtained with the help of Lemma \ref{lemma convergence of distance function} below, cf.\;\cite[Lemma 28.13]{maggi_2012}.
\medskip

Our setup is that $G\subset C(Z)$ is a globally perimeter minimizing cone with vertex $O$, a vertex of the ambient cone. Moreover, $\mathcal{H}^{k'}(S^{G}_{k,\Bar{\varepsilon}})>0$. In particular, by the very same arguments as in Step 2, there exist a point $O'\in C(Z)$, $O'\neq O$ and a sequence $r_i\downarrow 0$ such that
\begin{equation}
    \lim_{i \to \infty} \frac{\mathcal{H}^{k'}_\infty \left(B_{r_i}(O') \cap  S^{G}_{k,\Bar{\varepsilon}}\right)}{r_i^{k'}} \geq C_{k'}>0\, .
\end{equation}

Up to taking a subsequence that we do not relabel, we can assume that the sequence $(C(Z),\sfd_C/r_i,\mathcal{H}^N,O',G)$ converges to $(C(Z'),\sfd_{C'},\mathcal{H}^N,O'',H)$, where $(C(Z'),\sfd_{C'},\mathcal{H}^N)$ is an $\RCD(0,N)$ metric measure cone splitting an additional $\R$ factor with respect to $C(Z)$ and $H\subset C(Z')$ is a global perimeter minimizer. Moreover,
\begin{equation}
   \mathcal{H}^{k'} \left(B_{1}(O'') \cap  S^{H}_{k,\Bar{\varepsilon}}\right )>0\, .
\end{equation}
Consider the sequence of functions $f_i:C(Z)\to \R$ defined as
\begin{equation}
    f_i(z):=\frac{\sfd^2_C(O,z)-\sfd^{2}_C(O,O')}{r_i}\, ,
\end{equation}
that we view as functions on the rescaled metric measure space $(C(Z),\sfd_C/r_i,\mathcal{H}^N,O')$.\\
By Lemma \ref{lemma convergence of distance function} below, the functions $f_i$ converge to some splitting function $g:C(Z')\to \R$ in $H^{1,2}_{\mathrm{loc}}$, see \cite{AmbrosioHonda1} for the relevant background. Moreover $\Delta f_i$ converge to $0$ uniformly.
\smallskip

We claim that, for any function $\varphi \in \mathrm{LIP}(C(Z'))\cap W^{1,2}(C(Z'))$ it holds 
\begin{equation}\label{splitting condition for the perimeter minimizer}
    \int_H \nabla \varphi \cdot \nabla g \ \d \mathcal{H}^N = 0\, .
\end{equation}
To see this, let $\varphi_i \in \mathrm{LIP}(X_i)\cap W^{1,2}(X_i)$ converging $H^{1,2}$-strongly to $\varphi$ along the sequence $(C(Z),\sfd_C/r_i,\mathcal{H}^N,O')$, whose existence was shown in \cite{AmbrosioHonda1}. 
Then, using the Gauss-Green formula (Theorem \ref{thm:NuEGG}) and the characterization of cones in Lemma \ref{characterization of cones} we obtain
\begin{equation}
\label{gauss green on Gi}
\begin{aligned}
    0 &= \int_{\partial^* G} \varphi_i \nu^i_{G} \cdot \nabla_i f_i \ \dPer_i(G) \\
    & = - \int_{G} \nabla_i \varphi_i \cdot \nabla_i f_i \ \d \mathcal{H}^N - \int_{G}  \varphi_i \cdot \Delta_i f_i \ \d \mathcal{H}^N\, ,
\end{aligned}
\end{equation}
where the Hausdorff measure $\mathcal{H}^N$ is computed with respect to the rescaled distance $\sfd_C/r_i$.\\
By \eqref{eq:uselapest} below and \eqref{gauss green on Gi}
\begin{equation}
\label{convergence to zero of the approximants to the characterization of Per min split}
\begin{aligned}
     \int_{G} \nabla_i \varphi_i \cdot \nabla_i f_i \ \d \mathcal{H}^N = - \int_{G}  \varphi_i \cdot \Delta_i f_i \ \d \mathcal{H}^N \to 0\, .
\end{aligned}
\end{equation}
On the other hand, by \cite[Theorem 5.7]{AmbrosioHonda1}, it follows that
\begin{equation}
    \label{convergence of the approximants to the Per Min split}
    \int_{G} \nabla_i \varphi_i \cdot \nabla_i f_i \ \d \mathcal{H}^N \to \int_H \nabla \varphi \cdot \nabla  g \ \d \mathcal{H}^N\, .
\end{equation}
Combining \eqref{convergence to zero of the approximants to the characterization of Per min split} and \eqref{convergence of the approximants to the Per Min split} we obtain \eqref{splitting condition for the perimeter minimizer}; cf.\;\cite{SemolaGaussGreen,SemolaCutAndPaste} for analogous arguments.
\medskip

Our next goal is to use \eqref{splitting condition for the perimeter minimizer} to show that the perimeter minimizer $H$ splits a line in the direction of the ambient splitting induced by the splitting function $g$. 

Let us set $Y:=C(Z')=\R\times Y'$, and assume that $\R$ is the splitting induced by $g$.

Given any $\varphi\in W_\mathrm{loc}^{1,2}(Y)$ let us also denote $\varphi^{(t)}(y) := \varphi(t,y)$ and $\varphi^{(y)}(t) := \varphi(t,y)$. If $\varphi\in W_{\rm loc}^{1,2}(Y)$, then $\varphi^{(t)} \in W^{1,2}_\mathrm{loc}(Y')$ and $\varphi^{(y)} \in W_\mathrm{loc}^{1,2}(\R)$, for $\mathcal{L}^1$-a.e.\;$t$ and $\mathcal{H}^{N-1}$-a.e.\;$y$ respectively (see \cite{Gigli2018SobolevSO}). Up to the isomorphism given by the splitting induced by $g$, there holds
\begin{equation*}
    \nabla \varphi \cdot \nabla g (t,y)= \partial_t \varphi^{(y)}(t)\, ,\quad \text{for $\mathcal{H}^N$-a.e. $(t,y)\in Y$}\, .
\end{equation*}
Let $P_s$ denote the heat flow on $Y$. Then
\begin{equation}
    \label{heat flow of characteristic of H against derivative of test}
    \begin{aligned}
    \int_Y \mathrm{P}_s \chi_H (t,y) \nabla\varphi \cdot \nabla g (t,y) \, \d\mathcal{H}^N = \int_Y \mathrm{P}_s \chi_H (t,y) \partial_t \varphi^{(y)} (t) \, \d\mathcal{H}^N \\
    = \int_H \mathrm{P}_s \partial_t \varphi^{(y)} (t) \, \d\mathcal{H}^N = \int_H \partial_t (\mathrm{P}_s \varphi)^{(y)} (t)\, ,
    \end{aligned}
\end{equation}
where in the second equality we have used the self-adjointess of the heat flow and in the last equality we have used Lemma \ref{Heat flow and derivative in the splitting direction commute}. 

By \eqref{splitting condition for the perimeter minimizer} and \eqref{heat flow of characteristic of H against derivative of test} it follows that 
\begin{equation}
    \int_Y \mathrm{P}_s \chi_H (t,y) \nabla \varphi \cdot \nabla g (t,y) \, \d\mathcal{H}^N =\int_H \nabla (\mathrm{P}_s \varphi) \cdot \nabla g \, \d\mathcal{H}^N = 0\, .
\end{equation}
Since $\varphi \in \mathrm{LIP}(Y) \cap W^{1,2}(Y)$ is arbitrary, an elementary computation using Fubini's theorem and the splitting $Y=\R\times Y'$ shows that 
\begin{equation*}
    \partial_t (\mathrm{P}_s \chi_H)^{(y)}(t) = 0 \, \quad \text{for $\mathcal{L}^1$-a.e. $t\in\R$, for $\mathcal{H}^{N-1}$-a.e. $y\in Y'$}\, .
\end{equation*}
By the $L^1_\mathrm{loc}(Y)$ convergence of $\mathrm{P}_s \chi_H$ to $\chi_H$ for $s\downarrow 0$ and the closure of $H$, we conclude that $\chi_H^{(y)}$ is constant in $t$ for every $y\in Y'$. 
That implies the existence of a set $H' \subset Y'$ such that 
\begin{equation}
    \label{factorization of the characteristic of H}
    \chi_H (t,y) = \chi_{H'} (y) \, . 
\end{equation}
By Lemma \ref{lemma | Perimeter of cylinders}, $H' \subset Y'$ is a set of locally finite perimeter.\\ 
Let us show that $H'$ is a global perimeter minimizer, by following the classical Euclidean argument, cf.\;\cite[Lemma 28.13]{maggi_2012}. Suppose not. Then there exist $\varepsilon > 0$ and a set $H'_0 \subset Y'$ such that $H' \Delta H'_0 \subset \!\subset B_r(y)$ for some $r>0$ and $y\in Y'$, such that 
\begin{align}
    \mathrm{Per}(H'_0;B_r(y)) + \varepsilon \leq \mathrm{Per}(H';B_r(y))\, .
\end{align}
Let $t>0$. We define the sets
\begin{align}
   & I_t : = \R \setminus (-t,t) & H_0 : = \left(H'_0 \times (-t,t) \right) \cup (H' \times I_t)\, .
\end{align}

At this stage, we can use the formulas for the cut and paste of sets of finite perimeter (Theorem \ref{Thm:CutAndPastePer}), observe
that $H\Delta H_0 \subset B_r(y) \times (-t,t) := A$ and conclude by Lemma \ref{lemma | Perimeter of cylinders} that
\begin{align*}
    \mathrm{Per}(H_0;A) - \mathrm{Per}(H;A) &= 2t (\mathrm{Per} (H_0';B_r(y)) - \mathrm{Per}(H; B_r(y))) + 2\mathcal{H}^{N-1}(H_0' \Delta H')\\
    & \leq -2t\varepsilon + 2 \mathcal{H}^{N-1}(B_r(y))<0\, ,
\end{align*}
where we have chosen $t>0$ large enough so that $ \mathcal{H}^{N-1}(B_r(y)) < t\varepsilon $.\\ 
Therefore, $H'$ is a global perimeter minimizer, as we claimed.
\smallskip

If $k=0$, the above argument leads to a contradiction. Indeed we found a point in $\mathcal{S}^H_0$ such that some tangent space splits a line.

\medskip

\textbf{Step 4.}\\
If $k>0$, then it is straightforward to see that $(t,y)\in S_{k, \Bar{\varepsilon}}^{H}$ if and only if $y\in S_{k-1, \Bar{\varepsilon}}^{H'}$.\\
In particular, from the assumption that 
\begin{equation}
   \mathcal{H}^{k'} \left(B_{1}(O'') \cap  S^{H}_{k,\Bar{\varepsilon}}\right )>0\, ,
\end{equation}
we conclude that 
\begin{equation}
   \mathcal{H}^{k'-1} \left(B_{1}(O'') \cap  S^{H'}_{k-1,\Bar{\varepsilon}}\right )>0\, .
\end{equation}
Therefore the steps from 1 to 3 prove that if there exist an $\RCD(K,N)$ metric measure space $(X,\sfd,\mathcal{H}^N)$ and locally perimeter minimizing set $E\subset X$ such that for some $0\le k\le N-3$ it holds $\dim_{\mathcal{H}}(\mathcal{S}^E_{k})>k$, then there exist an $\RCD(0,N-1)$ space $(X',\sfd',\mathcal{H}^{N-1})$ and a locally perimeter minimizing set $E'\subset X'$ such that $\dim_{\mathcal{H}}(\mathcal{S}^{E'}_{k-1})>k-1$. The dimension reduction can be iterated a finite number of times until we reduce to the case $k=0$, that we already discussed above.
\end{proof}

\begin{proof}[Proof of Theorem \ref{thm:blowdown}]
First of all, up to modifying $E$ on a set of measure zero if necessary, we can (and will) assume that $E$ is open. 

\textbf{Step 1.}  Fix a point $x\in \partial E$.  We claim that there exists $C>1$ such that
\begin{align}
\frac{r^N}{C} \leq \mathcal{H}^N(E\cap B_r(x)) \leq C \; r^N,& \quad \text{for all } r>0, \label{eq:volEBr} \\
\frac{r^{N-1}}{C} \leq \Per(E;  B_r(x)) \leq C \; r^{N-1},& \quad \text{for all } r>0 \label{eq:PerEBr}.
\end{align}
Recall that an $\RCD(0,N)$ space is globally doubling (thanks to the Bishop-Gromov inequality \cite{Sturm2}) and satisfies a global Poincar\'e inequality \cite{RajalaPoincare}. Since, by assumption, $E$ mimimizes the perimeter on every metric ball then, by \cite[Theorem 4.2]{KinunnenJGA2013}, there exists a constant $\gamma_0>0$ (depending only on the doubling and Poincar\'e constants of $(X,\sfd,\mathcal{H}^N)$) such that 
\begin{equation}\label{eq:KinunnenVol}
\frac{\mathcal{H}^N(E\cap B_r(x))}{\mathcal{H}^N(B_r(x))}\geq \gamma_0 \quad \text{ and } \quad \frac{\mathcal{H}^N(B_r(x)\setminus E)}{\mathcal{H}^N(B_r(x))}\geq \gamma_0 \quad \text{for all  $r>0$ and $x\in \partial E$}.
\end{equation}
Recall that the ratio $\mathcal{H}^N(B_r(x))/r^N$ is monotone non-increasing by Bishop-Gromov inequality, it is bounded above by the value in $\R^N$ and it is bounded below by a positive constant thanks to the assumption \eqref{eq:AssEVG}. Hence  \eqref{eq:volEBr} follows from \eqref{eq:KinunnenVol}. The perimeter estimate \eqref{eq:PerEBr} follows from \eqref{eq:volEBr} and \cite[Lemma 5.1]{KinunnenJGA2013}.
\smallskip

\textbf{Step 2.} The argument is similar to those involved in the proof of Theorem \ref{thm: strat core} above and therefore we only sketch it.   Let $r_i\to \infty$ be any sequence such that $(X,\sfd/r_i,\mathcal{H}^N,x)$ converges to a tangent cone at infinity $(C(Z),\sfd_{C(Z)},\mathcal{H}^N,O)$ of $(X,\sfd,\mathcal{H}^N)$. By the Ahlfors regularity estimates \eqref{eq:volEBr}-\eqref{eq:PerEBr} and the compactness and stability \cite[Theorem 2.42]{WeakLaplacian}, the sequence  $(X,\sfd/r_i,\mathcal{H}^N,E,x)$ converges to $(C(Z),\sfd_{C(Z)},\mathcal{H}^N,O,F)$ for some non-empty perimeter minimizer $F\subset C(Z)$. At this stage, we are in position to apply Theorem \ref{Thm:MonotonicityBody} and obtain a perimeter minimizing cone in $C(Z)$, up to possibly taking an additional blow-down. 
\end{proof}

In the remainder of the section, we present some technical results that have been used in the proof of Theorem \ref{thm: strat core}.

\begin{lemma}\label{lemma: 2d}
Let $(Z,\sfd_Z,\mathcal{H}^2)$ be an $\RCD(0,2)$ metric measure cone and let $G\subset Z$ be a globally perimeter minimizing set, in the sense of Definition  \ref{def:LocPerMin}. Then one of the following two possibilities occur:
\begin{itemize}
    \item[i)] $(Z,\sfd_Z,\mathcal{H}^2)$ is isomorphic to $(\R^2,\sfd_{\mathrm{eucl}},\mathcal{H}^2)$ and $G$ is a half-plane;
    \item[ii)] $(Z,\sfd_Z,\mathcal{H}^2)$  is isomorphic to the half-plane $(\R^2_+,\sfd_{\mathrm{eucl}},\mathcal{H}^2)$ and $G$ is a quadrant.
\end{itemize}
\end{lemma}

\begin{proof}
    We distinguish two cases: if $Z$ has no boundary, then we prove that it is isometric to $\R^2$ and i) must occur; if $Z$ has non-empty boundary, then we prove that it is isometric to $\R^2_+$ and that ii) must occur.
    \medskip

    Let us assume that $(Z,\sfd_Z,\mathcal{H}^2)$ has empty boundary. Then, by \cite{LowDimKitabeppu}, $Z$ is isometric to a cone over $S^1(r)$ for some $0<r\leq1$. Moreover,  by Theorem \ref{thm:blowdown} there exists a blow-down of $G$ which is a global perimeter minimizing cone $C(A)$, with vertex in the origin and $A\subset S^1(r)$ connected. Indeed, it is elementary to check that if $A$ is not connected, then $C(A)$ is not locally perimeter minimizing. 
    
    Let $2\pi r \theta$ be the length of $A$, where $0<\theta<1$. Let $G'\subset Z$ be a set of finite perimeter such that $G' = G$ outside $B_1$ and $\partial G \cap B_1 $ is composed by the geodesic connecting the two points in $\partial G \cap \partial B_1 = \{x_1,x_2\}$. Such geodesic is contained in $B_1$ as can be verified through the explicit form of the metric. Using \eqref{distance on a cone}
\begin{equation}\label{perimeter comparison in cone over S1}
\begin{aligned}
    \mathrm{Per}(G'; B_1) &= \sqrt{2(1-\cos(\sfd_{Z'}(x_1,x_2)\wedge \pi))} \\
    & \leq 2 = \mathrm{Per}(G; B_1).
\end{aligned}
\end{equation}
Equality in \eqref{perimeter comparison in cone over S1} is achieved for $2\pi r \theta = \sfd_{S^1(r)}(x_1,x_2) \geq \pi$, that is for $1\geq r\theta\geq\frac{1}{2}$. Let us notice, by symmetry of $S^1(r)$, that we may suppose that $\theta \leq \frac{1}{2}$. Indeed, for every fixed $\theta$, we may find a comparison set with perimeter equal to the one constructed above corresponding to $1-\theta$. Hence equality is only achieved at $r=1$, $\theta = \frac{1}{2}$, corresponding to the case where $Z=\R^2$ and $C(A)$ is a half space. 
Notice that once we have established that $Z$ is isometric to $\R^2$, it is elementary that $G$ must be a half-space.
\medskip

In the case where $Z$ has non-empty boundary, by \cite{LowDimKitabeppu} again, $Z$ is isometric to a cone over a segment $[0,l]$ for some $0< l\le \pi$. The upper bound for the diameter is required in order for the cone to verify the $\CD(0,2)$ condition. We claim that it must hold $l=\pi$.\\
As above, by Theorem \ref{thm:blowdown}, there exists a blow-down of $G$ which is a global perimeter minimizing cone $C(A)$, with vertex in the origin and $A\subset [0,l]$ some set of finite perimeter. 
If $A$ is not connected, then it is elementary to check that $C(A)$ is not globally perimeter minimizing. Notice also that the complement of a global perimeter minimizer is a global perimeter minimizer.\\
By minimality and symmetry we can suppose $G=C([0,l')]$, for some $0<l'\leq \frac{l}{2}$. 
 By considering a suitably constructed competitor in $B_1$, let us show that the only possibility is that $Z$ is a half-space and  $C(A)$ is a quadrant. Consider the set $G'$ coinciding with $G$ outside of $B_1$ and whose boundary inside $B_1$ is the geodesic minimizing the distance between $\partial B_1 \cap \partial G$ and $\partial Z$. Then $\mathrm{Per}(G';B_1)\leq \mathrm{Per}(G;B_1)$, with equality achieved only if $l=\pi$ and $l'=\frac{\pi}{2}$.
As above, once established that $Z$ is isometric to $\R^2_+$, it is elementary to check that $G$ must be a quadrant.
\end{proof}

It is a standard fact that any blow-up of a cone centered at a point different from the vertex splits a line. For our purposes it is important to observe that the blow-up of the squared distance function from the vertex is indeed a splitting function in the blow-up of the cone. 

\begin{lemma}\label{lemma convergence of distance function}

Let $(X,\sfd,\m)$ be an $\RCD(N-2,N-1)$  space and let $(C(X),\sfd_{C(X)},\m_{C(X)})$ be the metric measure cone over $X$, with vertex $O\in C(X)$. Fix $p\in C(X)$ with $p\neq O$. Let $r_i\downarrow 0$ and consider the sequence of rescaled spaces $Y_i:=(C(X),\sfd_{C(X)}/r_i,\m_{C(X)}/\m(B_{r_i}(p)),p)$ converging in the pmGH topology to a tangent space $Y$ of $C(X)$ at $p$. Then the functions 
\begin{equation}
    f_i(\cdot):=\frac{\sfd_{C(X)}^2(O,\cdot)-\sfd_{C(X)}^2(O,P)}{r_i}\, ,
\end{equation}
viewed as functions $f_i:Y_i\to \R$, have Laplacians uniformly converging to $0$ and converge in $H^{1,2}_{\mathrm{loc}}$ to a splitting function  $g:Y\to \R$, up to the extraction of a subsequence.
\end{lemma}

\begin{proof}
Let us set
\begin{equation}
   f(\cdot):=\sfd_{C(X)}^2(O,\cdot)-\sfd_{C(X)}^2(O,P)\, ,  
\end{equation}
in order to ease the notation. On $C(X)$ it holds (see \cite{ConeMetric})
\begin{equation}
    \Delta f =2N\, ,\quad |\nabla f(x)|=2\sfd_{C(X)}(x,O)\, ,\quad \text{for a.e. on $x\in C(X)$}\, .
\end{equation}
By scaling, we obtain that
\begin{equation}\label{eq:uselapest}
    \Delta f_i =2Nr_i\, ,\quad |\nabla f(x)|=2\sfd_{C(X)}(x,O)\, ,\quad \text{for a.e. $x\in Y_i$}\, ,
\end{equation}
where it is understood that the Laplacian and the minimal relaxed gradient are computed with respect to the metric measure structure $(C(X),\sfd_{C(X)}/r_i,\m_{C(X)}/\m(B_{r_i}(p)),p)$. Notice that $x\mapsto 2\sfd_{C(X)}(x,O)$ is a $2r_i$-Lipschitz function on $Y_i$, by scaling.\\
Hence the functions $f_i:Y_i\to\R$ are locally uniformly Lipschitz, they satisfy $f_i(p)=0$, and they have Laplacians uniformly converging to $0$. Up to the extraction of a subsequence, thanks to a diagonal argument, we can assume that they converge locally uniformly and in $H^{1,2}_{\mathrm{loc}}$ to a function $g:Y\to \R$ in the domain of the local Laplacian, and that $\Delta f_i$ converge to $\Delta g$ locally weakly in $L^2$, thanks to \cite{AmbrosioHonda1,AmbrosioHonda2}. We claim that $g$ is a splitting function on $Y$, which amounts to say that $\Delta g=0$ and $|\nabla g|$ is constant almost everywhere and not $0$.
\medskip

The fact that $\Delta g=0$ follows from the weak convergence of the Laplacians and the identity $\Delta f_i=2Nr_i$ that we established above.\\
Analogously, employing the identity $|\nabla f_i(x)|=2\sfd_{C(X)}(x,O)$ a.e. on $Y_i$, and the local $W^{1,2}$ convergence of $f_i$ to $g$, it is immediate to check that $|\nabla g|=2\sfd(\cdot,O)$ a.e. on $Y$. 
\end{proof}

The next result relates the Heat flow on product spaces with one dimensional derivatives.

\begin{lemma}[Heat flow and derivative in the splitting direction commute]\label{Heat flow and derivative in the splitting direction commute}
Let $(X,\sfd,\m)$ be an $\RCD(K,\infty)$ space and let $X\times\R$ be endowed with the standard product metric measure space structure. Let $\varphi\in W^{1,2}(X\times\R)$. Then for every $s>0$ it holds
\begin{equation}
    \mathrm{P}_s \partial_t \varphi(x,t)= \partial_t (\mathrm{P}_s\varphi)(x,t)\, ,
\end{equation}
for $\m_X\otimes \mathcal{L}^1$-a.e. $(x,t)\in X\times \R$.

\end{lemma}

\begin{proof}
The statement follows from the tensorization of the Cheeger energy and of the heat flow for products of $\RCD(K,\infty)$ metric measure spaces, see for instance \cite{AmbrosioBakryEmery,Ambrosio_2014}, and from the classical commutation between derivative and heat semi-group on $\R$ endowed with the standard metric measure structure. 
\end{proof}

It is a well known fact of the Euclidean theory (see for instance  \cite{maggi_2012}) that the perimeter enjoys natural tensorization properties, when taking an isometric product by an $\R$ factor. The next lemma establishes the $\RCD$ counterpart of this useful property.

\begin{lemma}[Perimeter of Cylinders]\label{lemma | Perimeter of cylinders}
Let $(X,\sfd_X,\m_X)$ be an $\mathrm{RCD}(K,N)$ space and let $F\subset X$ be a Borel set. Under these assumptions, $E := F \times \R \subset X \times \R$ is a set of locally finite perimeter (where the product $X\times \R$ is endowed with the standard product metric measure structure), if and only if $F \subset X$ is a set of locally finite perimeter. Moreover, for any open set $A\subset X$ and for any $R>0$ it holds
\begin{equation}\label{cylinder identity perimeter}
    R\, \mathrm{Per}(F;A) = \mathrm{Per}(E;A\times[0,R])\, .
\end{equation}

\begin{proof}
By the very definition of perimeter it holds
\begin{align*}
    \mathrm{Per}(E, A\times[0,R]) = \inf_{(\varphi_i)_i} \left\{\liminf_{i\to \infty}\int_0^R \int_A \mathrm{lip} \, \varphi_i (t,x)\, \d\m_X\, \d t\right\}\, ,
\end{align*}
where the infimum is taken over all sequences $(\varphi_i)_i \subset \mathrm{LIP}_\mathrm{loc}(A\times[0,R])$ such that $\varphi_i \to \chi_E$ in $L^1_\mathrm{loc}(A\times[0,R])$. We are going to prove \eqref{cylinder identity perimeter} and the first part of the statement will follow immediately.
\medskip

\textbf{Step 1.} Let us start by showing the inequality
\begin{equation}\label{cylinder identity perimeter | first inequality}
    R\,\mathrm{Per}(F;A) \geq \mathrm{Per}(E;A\times[0,R])\, .    
\end{equation}
Let $(\psi_i)_i \subset \mathrm{LIP}_\mathrm{loc}(A)$ be a competitor for the perimeter of $F$ in $A$, i.e. $\psi_i \to \chi_F$ in $L^1_\mathrm{loc} (A,\m_X)$ and all the functions $\psi_i$ are locally Lipschitz. Define $\phi(t,x) := \psi(x)$ for $0\leq t\leq R$ and $x \in A$. Then, by Fubini's Theorem, $\{\phi_i\}_i$ is a competitor for the perimeter of $E$ in $A\times [0,R]$. Therefore, 
\begin{align*}
    \mathrm{Per}(F;A) & = \inf_{(\psi_i)_i} \left\{\liminf_{i\to \infty} \int_A \mathrm{lip} \, \psi_i (x)\, \d\m_X\right\} \\ 
    & = \frac{1}{R} \inf_{(\psi_i)_i} \left\{\liminf_{i\to \infty}\int_0^R \int_A \mathrm{lip} \, \phi^{(t)}_i (x) \, \d\m_X\, \d t\right\}\\
    & \geq \frac{1}{R} \inf_{(\varphi_i)_i} \left\{\liminf_{i\to \infty}\int_0^R \int_A \mathrm{lip} \, \varphi_i (t,x)\, \d\m_X\, \d t\right\} = \frac{1}{R} \mathrm{Per} (E; A\times [0,R])\, ,
\end{align*}
where the inequality follows from the fact that, on the right hand side, we are taking the infimum over a larger class.
\medskip

\textbf{Step 2.} We prove the opposite inequality in \eqref{cylinder identity perimeter}. \\
Let us fix $\varepsilon>0$. There exists a sequence $(\varphi_i)_i \subset \mathrm{LIP}_\mathrm{loc}(A\times[0,R])$ with $\varphi_i \to \chi_E$ in $L^1_\mathrm{loc}(A\times[0,R])$ such that 
\begin{equation}
   \liminf_{i\to \infty}\int_0^R \int_A \mathrm{lip} \, \varphi_i (t,x)\, \d x\, \d t \leq \mathrm{Per} (E; A\times [0,R]) + \varepsilon\, .
\end{equation}

It is straightforward to check that $\mathrm{lip}\,\varphi_i^{(t)} (x) \leq \mathrm{lip} \,\varphi_i (t,x)$ for every $(t,x) \in \R \times X$.\\ 
Moreover, the sequence $(\varphi_i^{(t)})_i$ is a competitor for the variational definition of the perimeter of $F$ in $A$ for $\mathcal{L}^1$-almost every $t$, by the coarea formula. Therefore, by Fatou's lemma,
\begin{align*}
    R \, \mathrm{Per}(F;A) &\leq \int_0^R \liminf_{i\to \infty} \int_A \mathrm{lip} \, \varphi^{(t)}_i (x) \, \d\m_X\, \d t \\
    & \leq \liminf_{i\to \infty}\int_0^R \int_A \mathrm{lip} \, \varphi_i^{(t)} (x)\, \d\m_X\, \d t \\
    & \leq \liminf_{i\to \infty}\int_0^R \int_A \mathrm{lip} \, \varphi_i(t,x)\, \d\m_X\, \d t \leq \mathrm{Per}(E;A\times[0,R]) + \varepsilon\, .
\end{align*}
Since $\varepsilon > 0$ was arbitrary, we conclude.
\end{proof}
\end{lemma}

{
\small{ 
\printbibliography}
}
\end{document}